\documentclass[a4paper]{amsart}

\usepackage{amsmath,amstext,amssymb,mathrsfs,amscd,amsthm,indentfirst, bbm}
\usepackage{amsfonts}

\usepackage{tikz, tikz-cd}
\usetikzlibrary{matrix,calc,positioning,arrows,decorations.pathreplacing,decorations.markings,patterns}
\tikzset{->-/.style={decoration={
			markings,
			mark=at position #1 with {\arrow{>}}},postaction={decorate}}}
\tikzset{-<-/.style={decoration={
					markings,
					mark=at position #1 with {\arrow{<}}},postaction={decorate}}}

\usepackage{booktabs}
\usepackage{verbatim}
\usepackage{enumerate}

\usepackage{graphicx}

\DeclareFontFamily{T1}{cbgreek}{}
\DeclareFontShape{T1}{cbgreek}{m}{n}{<-6>  grmn0500 <6-7> grmn0600 <7-8> grmn0700 <8-9> grmn0800 <9-10> grmn0900 <10-12> grmn1000 <12-17> grmn1200 <17-> grmn1728}{}
\DeclareSymbolFont{quadratics}{T1}{cbgreek}{m}{n}
\DeclareMathSymbol{\qoppa}{\mathord}{quadratics}{19}
\DeclareMathSymbol{\Qoppa}{\mathord}{quadratics}{21}

\newcommand{\bC}{\mathbb{C}}

\newcommand{\bF}{\mathbb{F}}

\newcommand{\bL}{\mathbb{L}}

\newcommand{\bN}{\mathbb{N}}

\newcommand{\bQ}{\mathbb{Q}}
\newcommand{\bR}{\mathbb{R}}

\newcommand{\bZ}{\mathbb{Z}}

\newcommand{\cL}{\mathcal{L}}

\newcommand\lra{\longrightarrow}

\newcommand\colim{\operatorname*{colim}}
\newcommand\hocolim{\operatorname*{hocolim}}
\newcommand\holim{\operatorname*{holim}}

\newcommand{\Hom}{\mathrm{Hom}}

\newcommand{\Aut}{\mathrm{Aut}}

\renewcommand{\epsilon}{\varepsilon}

\newcommand{\Sq}{\mathrm{Sq}}

\newcommand{\sign}{\mathrm{sign}}
\newcommand{\Mock}{\mathrm{Mock}}
\newcommand{\E}{\mathrm{E}}
\newcommand{\KO}{\mathrm{KO}}
\newcommand{\K}{\mathrm{K}}
\newcommand{\MSO}{\mathrm{MSO}}
\newcommand{\MSPL}{\mathrm{MSPL}}

\newcommand{\MSTop}{\mathrm{MSTop}}
\newcommand{\LL}{\mathrm{L}}
\renewcommand{\H}{\mathrm{H}}
\newcommand{\GW}{\mathrm{GW}}
\newcommand{\Sph}{\mathrm{S}}

\mathchardef\ordinarycolon\mathcode`\:
\mathcode`\:=\string"8000
\begingroup \catcode`\:=\active
  \gdef:{\mathrel{\mathop\ordinarycolon}}
\endgroup

\usepackage{amsthm}
\theoremstyle{plain}

\newtheorem{theorem}{Theorem}[section]

\newtheorem{proposition}[theorem]{Proposition}
\newtheorem{lemma}[theorem]{Lemma}

\newtheorem{corollary}[theorem]{Corollary}

\theoremstyle{definition}
\newtheorem{definition}[theorem]{Definition}

\newtheorem{example}[theorem]{Example}

\theoremstyle{remark}

\newtheorem{remark}[theorem]{Remark}
\newtheorem*{remark*}{Remark}

\numberwithin{equation}{section}

\usepackage[hypertexnames=false]{hyperref}

\title{The family signature theorem}

\author{Oscar Randal-Williams}
\email{o.randal-williams@dpmms.cam.ac.uk}
\address{Centre for Mathematical Sciences\\
Wilberforce Road\\
Cambridge CB3 0WB\\
UK}

\begin{document}
\begin{abstract}
We discuss several versions of the Family Signature Theorem: in rational cohomology using ideas of Meyer, in $KO[\tfrac{1}{2}]$-theory using ideas of Sullivan, and finally in symmetric $L$-theory using ideas of Ranicki. Employing recent developments in Grothendieck--Witt theory, we give a quite complete analysis of the resulting invariants. As an application we prove that the signature is multiplicative modulo 4 for fibrations of oriented Poincar{\'e} complexes, generalising a result of Hambleton, Korzeniewski and
Ranicki, and discuss the multiplicativity of the de Rham invariant.
\end{abstract}
\maketitle

\vspace{-1ex}

\tableofcontents

\section{Introduction}

For an oriented family of manifolds $\pi : E \to B$, a family signature theorem is an equation between two classes in the (generalised) cohomology of $B$. The first should be local in the total space $E$, and arise as the image under a fibre-integration map of a class on $E$ which is locally defined. The second should depend only on the local coefficient system on $B$ given by the middle cohomology of the fibres, equipped with its intersection (or linking) form. When the fibres have dimension $4k$ and $B$ is a point, it should reduce to Hirzebruch's signature theorem.

We will discuss several implementations of this idea, in increasing levels of sophistication: firstly in rational cohomology, then in real $K$-theory localised away from 2, and finally in symmetric $L$-theory of the integers. Our main contribution is to do so in what we feel is the correct generality, by interpreting ``family of manifolds'' to mean ``topological block bundle'', i.e.\ block bundles with topological manifold fibres. Under more restrictive conditions the results are (well) known: for smooth fibre bundles the rational cohomology statement is due to Atiyah \cite[eq.\ (4.3)]{Atiyah}, and for PL fibre bundles the $L$-theory statement was given by L{\"u}ck and Ranicki \cite[p.\ 184]{LuckRanicki} and its proof outlined. As the statement for smooth block bundles in rational cohomology has seen some recent use \cite{EbertReinhold}, and the author will soon need the statement for topological fibre bundles in $L$-theory \cite{GRWPontAlgInd}, it seems appropriate to give a detailed stand-alone account.

The highlights are: the three forms of the Family Signature Theorem (Theorems \ref{thm:FSTQ}, \ref{thm:FSTSullivan}, \ref{thm:FSTLThy}); triviality of the family signature for families of odd-dimensional manifolds (Theorem \ref{cor:OddFSIsTrivial});  multiplicativity of the signature modulo 4 for local systems over topological manifolds (Corollary \ref{cor:SigMultMeyer}) and for fibrations of Poincar{\'e} complexes (Corollary \ref{cor:MultSigPoinc}); an analysis of the (non)multiplicativity of the de Rham invariant (Section \ref{sec:dRmult}).

\vspace{2ex}

\noindent\textbf{Acknowledgements}. I would like to thank J.\ Ebert, F.\ Hebestreit, and especially M.\ Land for useful discussions, and the anonymous referee for their helpful remarks. I have been supported by the ERC under the European Union's Horizon 2020 research and innovation programme (grant agreement No.\ 756444) and by a Philip Leverhulme Prize from the Leverhulme Trust.

\section{Meyer}

We first give an exposition of the work of Meyer \cite{Meyer}, and use it to prove a form of the family signature theorem for topological block bundles in rational cohomology. We also use it to prove the multiplicativity of the signature modulo 4 for local systems of $(-1)^n$-symmetric forms over a topological manifold.

\subsection{Twisted signatures}

Suppose that $n \in \bN$, $H_\bR$ is a real vector space and $\lambda : H_\bR \otimes H_\bR \to \bR$ is a $(-1)^n$-symmetric bilinear form which is nondegenerate, i.e.\ such that the adjoint map $\lambda^\mathrm{ad} : H_\bR \to \mathrm{Hom}_\bR(H_\bR, \bR)$ is an isomorphism. Let $\Aut(H_\bR, \lambda) \leq GL(H_\bR)$ denote the subgroup of those automorphisms of $H_\bR$ which preserve the form $\lambda$, considered as a discrete group. There is a corresponding flat vector bundle, or local coefficient system, $\mathcal{H}_\bR \to B\Aut(H_\bR, \lambda)$, equipped with a nondegenerate $(-1)^n$-symmetric fibrewise bilinear form.

If $M^{4k-2n}$ is an oriented closed manifold and $f : M \to B\Aut(H_\bR, \lambda)$ is a map, then the bilinear form
\begin{align*}
H^{2k-n}(M ; f^*\mathcal{H}_\bR) \otimes H^{2k-n}(M ; f^*\mathcal{H}_\bR) &\overset{\smile}\lra H^{4k-2n}(M ; f^*(\mathcal{H}_\bR \otimes \mathcal{H}_\bR))\\
& \overset{\lambda}\lra H^{4k-2n}(M;\bR) \overset{\int_M}\lra \bR,
\end{align*}
is symmetric because the cup product and $\lambda$ are either both symmetric or both antisymmetric. We may therefore take its signature
$$\sigma(M; f) := \sigma(H^{2k-n}(M ; f^*\mathcal{H}_\bR)) \in \bZ,$$
and Meyer proves that this depends only on the oriented cobordism class of the map $f : M^{4k-2n} \to B\Aut(H_\bR, \lambda)$. This is the \emph{twisted signature} of $(M;f)$.

\subsection{Meyer's formula for twisted signatures}\label{sec:MeyerFormula}

Given the flat vector bundle $\mathcal{H}_\bR \to B\Aut(H_\bR, \lambda)$ and the nondegenerate $(-1)^n$-symmetric fibrewise bilinear form induced by $\lambda$, we may choose a Riemannian metric $\langle -, - \rangle$ on this bundle and hence define a fibrewise operator $A$ by $\lambda(x,y) = \langle x, A y \rangle$ satisfying $A^* = (-1)^n A$, and use this to form a complex $K$-theory class $\xi \in K^0(B\Aut(H_\bR, \lambda))$ as follows (see \cite[\S 1]{Meyer}, \cite[pp.\ 478-9]{AtiyahSingerIII}):
\begin{enumerate}[(i)]
\item If $n$ is even then $A$ is self-adjoint and its positive and negative eigenspaces give a decomposition $\mathcal{H}_\bR^+ \oplus \mathcal{H}_\bR^-$ of $\mathcal{H}_\bR$, and we set
$$\xi := (\mathcal{H}_\bR^+ - \mathcal{H}_\bR^-) \otimes \bC \in K^0(B\Aut(H_\bR, \lambda)).$$

\item If $n$ is odd then $A$ is skew-adjoint and $J := A/\sqrt{A A^*}$, formed using the positive square root of $AA^*$,  determines a complex structure on $\mathcal{H}_\bR$, and we set
$$\xi := \overline{\mathcal{H}}_\bR - \mathcal{H}_\bR \in K^0(B\Aut(H_\bR, \lambda)),$$
where $\overline{\mathcal{H}}_\bR$ denotes the complex conjugate bundle.
\end{enumerate}
As the Riemannian metric is not adapted to the flat structure of $\mathcal{H}_\bR$, these virtual vector bundles need not be flat and can have interesting Chern classes. For $f : M \to B\Aut(H_\bR, \lambda)$ with $M$ an oriented smooth manifold Meyer establishes the following formula 
\begin{equation}\label{eq:MayerFormula}
\sigma(M; f) = \int_M \mathrm{ch}(\psi^2(f^*\xi)) \cdot \mathcal{L}(TM),
\end{equation}
by applying the index theorem to the signature operator of $M$ twisted in a certain way by the bundle $f^*\mathcal{H}_\bR$, using its nondegenerate $(-1)^n$-symmetric fibrewise bilinear form. Here $\mathrm{ch}$ denotes the Chern character, $\psi^2$ denotes the second Adams operation, and $\mathcal{L}$ denotes the Hirzebruch $L$-class.

\begin{remark}\label{rem:RealXi}
For later use, we observe that the complex $K$-theory classes $\xi$ have refinements to real $K$-theory classes $\xi_\bR$. Namely
\begin{enumerate}[(i)]
\item If $n$ is even then $\xi$ is visibly the complexification of
$$\xi_\bR := \mathcal{H}_\bR^+ - \mathcal{H}_\bR^- \in KO^{0}(B\Aut(H_\bR, \lambda)).$$

\item If $n$ is odd then writing $c : KO^i(-) \to K^i(-)$ and $r : K^i(-) \to KO^i(-)$ for the complexification and realification maps, and $b \in K^{-2}(*)$ for the Bott class, we have $b \cdot \xi = c(r(b \cdot \overline{\mathcal{H}}_\bR))$, so $b \cdot \xi$ is the complexification of
$$\xi_\bR := r(b \cdot \overline{\mathcal{H}}_\bR) \in KO^{-2}(B\Aut(H_\bR, \lambda)).$$
\end{enumerate}
\end{remark}

\subsection{Divisibility and multiplicativity of the signature}

It does not seem to be well-known that the discussion so far can be used to establish multiplicativity of the signature modulo 4, for example recovering and in fact strengthening the main theorem of \cite{HKR}. Although this is not our main goal, we take a brief excursion in this and the following subsection to explain how. Related ideas will arise in Section \ref{sec:AddendaEven} where we investigate the delicate 2-local structure of the $L$-theoretic family signature theorem.

\begin{lemma}\label{lem:MayerTopological}
The identity \eqref{eq:MayerFormula} holds even if $M$ is an oriented topological manifold.
\end{lemma}
\begin{proof}
Interpreting $\mathcal{L}(TM)$ as the topological Hirzebruch $L$-class the two sides define homomorphisms $MSTop_{4k-2n}(B\Aut(H_\bR, \lambda)) \to \bQ$, which are equal when precomposed with $MSO_{4k-2n}(B\Aut(H_\bR, \lambda)) \to MSTop_{4k-2n}(B\Aut(H_\bR, \lambda))$. But the latter map is an isomorphism after rationalising, as $BSO \to BSTop$ is a rational equivalence, by \cite[p.\ 246 eq. (5)]{KS} and the finiteness of the groups $\Theta_n$ of homotopy $n$-spheres for $n \geq 5$.
\end{proof}

\begin{corollary}\label{cor:SigMultMeyer}
If $(H_\bR, \lambda)$ is a nondegenerate $(-1)^n$-symmetric bilinear form and $f : M^{4k-2n} \to B\Aut(H_\bR, \lambda)$ is a map from an oriented topological manifold, then
$$\sigma(M; f) \equiv \sigma(M) \cdot \sigma(H_\bR, \lambda) \mod 4.$$
\end{corollary}
\begin{proof}
We rely on the formula \eqref{eq:MayerFormula}, which holds in this situation by the previous lemma. If the manifold $M$ is smooth the the polynomials exhibiting $\mathcal{L}_i(TM)$ in terms of the integral cohomology classes $p_j(TM)$ are defined over the 2-local integers $\bZ_{(2)}$, as the coefficients of the power series $x/\mathrm{tanh}(x)$ lie in this ring\footnote{This may be seen as follows. We have $x/\mathrm{tanh}(x) = 1 + \sum_{i=1}^\infty \frac{2^{2i} B_{i}}{(2i)!} x^{2i}$ for Bernoulli numbers $B_i$ defined by $\frac{x}{e^x-1} = 1 -\tfrac{1}{2}x + \sum_{i = 1}^\infty \tfrac{B_i}{(2i)!} x^{2i}$. By the von Staudt--Clausen theorem each $B_i$ has 2-adic valuation exactly $-1$, and by Legendre's theorem $(2i)!$ has 2-adic valuation $\leq 2i-1$: thus $\frac{2^{2i} B_{i}}{(2i)!} $ has 2-adic valuation $\geq 2i +(-1) - (2i-1) = 0$, so it is a 2-local integer.}, and so we have refinements $\mathcal{L}_i(TM) \in H^{4i}(M;\bZ_{(2)})$. If $M$ is a topological manifold then we cannot make this argument, but its conclusion is nonetheless true by work of Morgan and Sullivan \cite{MorganSullivan} (see Section 7, and the remark at the end of that section). As the classes $\mathcal{L}_i(TM)$ are 2-integral, and as $\mathrm{ch}_0(\psi^2(f^*\xi)) =  f^*\mathrm{ch}_0(\xi) = \sigma(H_\bR, \lambda)$ if $n$ is even, to conclude the argument it suffices to show that $\mathrm{ch}_i(\psi^2(\xi))$ are 2-integral and 2-integrally divisible by 4 for all $i>0$.

For any complex vector bundle $V$ we have 
$$\mathrm{ch}_i(\psi^2(V)) = 2^i \mathrm{ch}_i(V) = \tfrac{2^i}{i!} \bar{p}_i(c_1(V), c_2(V), \ldots, c_i(V))$$
where $\bar{p}_i$ is the polynomial over $\bZ$ expressing the power sum polynomial $t_1^i + t_2^i + t_3^i + \cdots$ in terms of elementary symmetric polynomials. By Legendre's formula $\tfrac{2^i}{i!}$ is 2-integral and 2-integrally divisible by 2.
There are now cases depending on the parity of $n$. If $n$ is odd then $\mathrm{ch}_i(\xi) = \mathrm{ch}_i(\overline{\mathcal{H}}_\bR - \mathcal{H}_\bR) = (1-(-1)^i) \mathrm{ch}_i(\overline{\mathcal{H}}_\bR)$ which vanishes for $i$ even and is $2\mathrm{ch}_i(\overline{\mathcal{H}}_\bR)$ for $i$ odd. As $\mathrm{ch}_i(\overline{\mathcal{H}}_\bR)$ is 2-integrally divisible by 2 by the previous paragraph, this finishes the argument in this case.

If $n$ is even then $2c_i(\xi)=0$ for $i$ odd, as $\xi$ is the complexification of a real vector bundle. This implies that $2\bar{p}_i(c_1(\xi), c_2(\xi), \ldots, c_i(\xi))$ is zero for $i$ odd, as then each monomial much contain some odd Chern class, and we will now show that $2\bar{p}_{2i}(c_1(\xi), c_2(\xi), \ldots, c_{2i}(\xi))$ is 2-integrally divisible by 4. Writing $2i = 2^r \cdot s$ with $r \geq 1$ and $s$ odd, we have
$$t_1^{2i} + t_2^{2i} + t_3^{2i} + \cdots \equiv (t_1^s + t_2^s + t_3^s + \cdots)^{2^r} \mod 2$$
and so $2\bar{p}_{2i}(c_1(\xi), c_2(\xi), \ldots, c_{2i}(\xi)) \equiv 2\bar{p}_s(c_1(\xi), c_2(\xi), \ldots, c_{2i}(\xi))^{2^r} \mod 4$. But as $s$ is odd the right-hand side vanishes by the case discussed above, which finishes the argument in this case.
\end{proof}

\subsection{Signatures of fibrations of Poincar{\'e} complexes}\label{sec:PoincareFib}

We consider Poincar{\'e} complexes in the sense of Wall \cite{WallPoincare}, i.e.\ finitely-dominated CW-complexes enjoying Poincar{\'e} duality with respect to a twisted integral fundamental class and all systems of local coefficients. Suppose that $F^{d} \to E \to B^{4k-d}$ is a fibration of finitely-dominated spaces with Poincar{\'e} base and fibre (and hence Poincar{\'e} total space too \cite[Corollary F]{KleinDualising}, \cite{Gottlieb}), which is oriented in the sense that $B$ is oriented, $F$ is orientable, and the local coefficient system $\mathcal{H}^d(F;\bZ)$ is trivialised: this induces an orientation of $E$.  
Let
$$E_2^{p,q} = H^p(B ; \mathcal{H}^q(F;\bR)) \Longrightarrow H^{p+q}(E;\bR)$$
denote the Serre spectral sequence for this fibration. If $d=2n$ then there is a map
$$\phi : B \lra B\Aut(H^n(F ; \bR), \lambda)$$
given by the action of the fundamental groupoid of $B$ on the middle cohomology of the fibres, which preserves the $(-1)^n$-symmetric form $\lambda$ given by cup product. The definition of the twisted signature above did not really use that $M$ is a manifold, only that it has Poincar{\'e} duality with all systems of local coefficients: thus we can define $\sigma(B; \phi)$ in the same way. Meyer shows that there is an identity
\begin{equation}\label{eq:MeyerPoincare}
\sigma(E) = \begin{cases}
\sigma(B; \phi) & \text{if $d=2n$}\\
0 & \text{if $d$ is odd}.
\end{cases}
\end{equation}

\begin{remark}\label{rem:MultSig}
Combining \eqref{eq:MeyerPoincare} with Corollary \ref{cor:SigMultMeyer} shows that $\sigma(E) \equiv \sigma(B) \cdot \sigma(F) \mod 4$ as long as $B$ is homotopy equivalent to a topological manifold, generalising the main theorem of \cite{HKR}. This result also follows from Korzeniewski's thesis \cite[Theorem 7.2]{Korzeniewski}, which allows the base to be a finite Poincar{\'e} complex having trivial Whitehead torsion (and a topological manifold is an example of this). We will give a further strengthening of this result as Corollary \ref{cor:MultSigPoinc}, which allows $B$ to be an arbitrary Poincar{\'e} complex.
\end{remark}

Meyer proves \eqref{eq:MeyerPoincare} in two steps: Firstly $\sigma(E)$ is related to the signature of $E_2^{*,*}$ taken with respect to the form
$$E_2^{*,*} \otimes E_2^{*,*} \overset{\smile}\lra E_2^{*,*} \lra \bR[4k-d, d],$$
where the latter map denotes the projection to $E_2^{4k-d,d} = H^{4k-d}(B ; \mathcal{H}^d(F;\bR)) = \bR$. This is done by  (i) showing that the signature is unchanged by passing from one page of the spectral sequence to the next, so is the same as the signature of $E_\infty^{*,*}$ with the analogous form (this is \cite[Satz I.1.4]{Meyer}), (ii) observing that the signature of $E_\infty^{*,*}$ is identified with $\sigma(E)$, as signatures are unchanged under passing to associated gradeds. Secondly, by recognising $L := \bigoplus_{q > d/2} E_2^{*, q}$ as a sublagrangian of $E_2^{*,*}$ with $L^\perp = \bigoplus_{q \geq d/2} E_2^{*, q}$, this signature is the same as that of the induced form on $L^\perp/L$, which is trivial if $d$ is odd and is the form described above on $H^{2k-d/2}(B; \mathcal{H}^{d/2}(F;\bR))$ if $d$ is even.

\begin{remark}
Meyer assumes various additional hypotheses, most notably that $B$ and $F$ are homology manifolds. This is because he wishes to allow (non locally constant) sheaf coefficients. For locally constant coefficients being Poincar{\'e} complexes suffices for his argument to go through.
\end{remark}

\subsection{The family signature theorem over $\bQ$}\label{sec:FSTQ}

We may rationalise the twisted signature map $\sigma : \Omega_*^{\mathrm{fr}}(B\Aut(H_\bR, \lambda)) \to \bZ$ on framed bordism, then use that rational framed bordism is naturally isomorphic to rational homology, and that rational cohomology is dual to rational homology, to define rational cohomology classes
$$\sigma_{4k-2n} \in H^{4k-2n}(B\Aut(H_\bR, \lambda);\bQ).$$
In other words, if $f : W^{4k-2n} \to B\Aut(H_\bR, \lambda)$ is a map from a stably framed manifold, then $\sigma(W; f) = \int_W f^* \sigma_{4k-2n}$. It follows from \eqref{eq:MayerFormula}, applied to all maps $f : W^{4k-2n} \to B\Aut(H_\bR, \lambda)$ from stably framed manifolds (which have $\mathcal{L}(TW)=1$), that
\begin{equation}\label{eq:SigmaChern}
\sigma_{4k-2n} = \mathrm{ch}_{2k-n}(\psi^2(\xi)) = 2^{2k-d/2} \mathrm{ch}_{2k-d/2}(\xi).
\end{equation}

The family signature theorem---in the generality we wish to discuss it in this note---concerns oriented topological block bundles $\pi: E \to |K|$ with $d$-dimensional fibres, as described in \cite[Section 2.3]{HLLRW}. There the base $|K|$ is taken to be the realisation of a simplicial complex, but one can equally well take it to be the realisation of a semi-simplicial set. This notion then includes the universal block bundle described in \cite[Definition 2.3.1]{HLLRW} and the proceeding discussion. In \cite[Section 2.4]{HLLRW} it is shown that there is an associated stable vertical tangent microbundle, given by a map
$$T_\pi^s E : E \lra BTop,$$
and we write $\cL_k(T_\pi^s E)$ for the pullback of the $k$th Hirzebruch $L$-class along this map. An oriented block bundle also has an associated fibre-integration, or Gysin, map $\int_\pi : H^*(E) \to H^{*-d}(|K|)$ as discussed in \cite[Section 4.1]{HLLRW}, and the family signature theorem is then as follows.

\begin{theorem}[Family signature theorem over $\bQ$]\label{thm:FSTQ}
Let $\pi: E \to |K|$ be an oriented topological block bundle with fibre $F^d$. If $d=2n$ let $\phi : |K| \to B\Aut(H^n(F;\bR), \lambda)$ classify the local coefficient system $\mathcal{H}^{n}(F;\bR)$ over $|K|$ with the $(-1)^n$-symmetric fibrewise bilinear form given by cup product. Then
$$\int_\pi \mathcal{L}_k(T_\pi^s E) = \begin{cases}
\phi^*(\sigma_{4k-d}) & \text{if $d$ is even}\\
0 & \text{if $d$ is odd}
\end{cases} = \begin{cases}
2^{2k-d/2}\mathrm{ch}_{2k-d/2}(\phi^*(\xi)) & \text{if $d$ is even}\\
0 & \text{if $d$ is odd.}
\end{cases}$$
\end{theorem}
\begin{proof}
As rational cohomology is dual to rational framed bordism, by naturality it is enough to consider the case where $|K|=B^{4k-d}$ is a stably framed smooth manifold of dimension $4k-d$, where the cohomology classes in question are top-dimensional and are hence determined by their integrals over $B$. 

As $B$ is a stably framed smooth manifold then by \cite[Lemma 2.5.2]{HLLRW} $E$ is a topological manifold and its stable tangent microbundle satisfies 
$$TE \cong_s \pi^*(TB) \oplus T_\pi^s E \cong_s \bR^{2k-d} \oplus T_\pi^s E$$
so $\mathcal{L}_k(T_\pi^s E) = \mathcal{L}_k(T E)$ by multiplicativity of the total $L$-class, and hence
$$\int_B \int_\pi \mathcal{L}_k(T_\pi^s E) = \int_E \mathcal{L}_k(T E) = \sigma(E)$$
by the defining property of the (topological) Hirzebruch $L$-classes. A block bundle determines a fibration with homotopy equivalent fibre, total space, and base, so by \eqref{eq:MeyerPoincare} we have $\sigma(E)=0$ if $d$ is odd and $\sigma(E) = \sigma(B; \phi)$ if $d$ is even. To obtain the first identity observe that as $B$ is stably framed the definition of $\sigma_{4k-d}$ gives $\sigma(B; \phi) = \int_B \phi^*(\sigma_{4k-d})$. The second identity then follows from \eqref{eq:SigmaChern}.
\end{proof}

\section{Sullivan}

We now explain how Meyer's work can be combined with ideas of Sullivan \cite{SullivanMIT} to obtain a $\bZ[\tfrac{1}{2}]$-integral form of the family signature theorem, formulated not in ordinary cohomology but in the generalised cohomology theory $KO[\tfrac{1}{2}]$, real $K$-theory localised away from 2. Useful later references for these ideas are \cite{MorganSullivan} (especially Section 1) and \cite{MadsenMilgram} (especially Chapter 4.B).

\subsection{Cobordism, coefficients, and $K$-theory}\label{sec:CobCoeffKthy}
Coefficients in an abelian group $A$ can be introduced into the generalised homology theory represented by a spectrum $\E$ by the device of smashing with the Moore spectrum $\mathrm{M}A$, i.e.\ setting $E_*(-;A) := E_*(\mathrm{M}A \wedge -) = \pi_*(\E \wedge \mathrm{M}A \wedge -)$. There is a corresponding cohomology theory, represented by $\E \wedge \mathrm{M}A$. When $A$ is a localisation of $\bZ$, i.e. a subring of $\bQ$, the construction of $\mathrm{M}A$ as a mapping telescope shows that $E_*(-;A) = E_*(-) \otimes_\bZ A$, and similarly for the cohomology theory on finite complexes (but not in general).

Write $\KO[\tfrac{1}{2}] := \KO \wedge \mathrm{M} \bZ[\tfrac{1}{2}]$ for the spectrum representing real $K$-theory localised away from 2, with homotopy ring $\pi_*(\KO[\tfrac{1}{2}]) = \pi_*(\KO) \otimes \bZ[\tfrac{1}{2}] = \bZ[\tfrac{1}{2}][a^{\pm 1}]$, for $a$ the class of degree 4 which under the complexification map is sent to $b^2$ with $b \in \pi_2(\K[\tfrac{1}{2}])$ the Bott class. There is an orientation (\cite[p.\ 201]{SullivanMIT})
$$\Delta : \MSO \lra \KO[\tfrac{1}{2}]$$
whose Pontrjagin character satisfies $\mathrm{ph}(\Delta) = \mathcal{L}^{-1} \cdot u \in H^*(\MSO;\bQ)$, where $u$ is the cohomological Thom class and $\cL \in H^*(BSO;\bQ)$ is the total Hirzebruch $L$-class. These conventions are arranged so that on homotopy groups the map
$$\Delta_* : MSO_* \lra KO[\tfrac{1}{2}]_* = \bZ[\tfrac{1}{2}][a^{\pm 1}]$$
is given by $\Delta_*([M^{4k}]) = \sigma(M) \cdot a^k$. Sullivan shows, using the results of Conner--Floyd, that the induced map $MSO_*(-) \otimes_{MSO_*}\bZ[\tfrac{1}{2}][a^{\pm 1}] \to KO[\tfrac{1}{2}]_*(-)$ is an isomorphism of generalised homology theories. This can also be deduced from Landweber's exact functor theorem \cite[Example 3.4]{Landweber}. Using the way we have introduced coefficients, and re-writing slightly, it follows that there is an isomorphism 
\begin{equation}\label{eq:ConnerFloyd}
MSO_{4*}(- ; A) \otimes_{MSO_{4*}} \bZ[\tfrac{1}{2}] \overset{\sim}\lra KO_0(- ; A \otimes \bZ[\tfrac{1}{2}])
\end{equation}
for any $A$.

Sullivan applies this to show that for $X$ a finite complex, the data of a class $\Phi \in KO^0(X ; \bZ[\tfrac{1}{2}])$ is equivalent to the data of morphisms $\phi_\bQ$ and $\phi_k$ for each odd $k$ such that
\begin{equation*}
\begin{tikzcd}
MSO_{4*}(X) \arrow[rr, "\phi_\bQ"] \dar & & \bZ[\tfrac{1}{2}] \dar\\
MSO_{4*}(X; \bZ/k) \rar{\phi_k}  & \bZ/k \rar & \bQ/\bZ
\end{tikzcd}
\end{equation*}
commutes, the $\phi_k$ are compatible under divisibility, and such that $\phi_\bQ$ and $\phi_k$ satisfy 
\begin{equation}\label{eq:Product}
\phi( W^{4m} \times M^{4n} \overset{\pi_W}\to W \overset{f}\to X) = \phi(W^{4m} \overset{f}\to X) \cdot \sigma(M^{4n})
\end{equation}
for all $[W,f] \in MSO_{4*}(X)$. Here one thinks of $\bQ/\bZ$ as given by $\colim_{\text{all }k} \bZ/k$. A $\Phi$ determines such maps $\phi$ by applying \eqref{eq:ConnerFloyd} and then evaluating the resulting $KO[\tfrac{1}{2}]$-homology class on $\Phi$. By replacing $X$ by its suspensions, we get a similar descriptions of elements of $KO^d(X;\bZ[\tfrac{1}{2}])$.

\subsection{Signatures of $\bZ/k$-manifolds}

Oriented cobordism with $\bZ/k$-coefficients, $MSO_n(X;\bZ/k)$, has an interpretation as cobordism classes of smooth $n$-dimensional singular $\bZ/k$-manifolds oveer $X$ \cite[\S 1]{MorganSullivan}. An oriented $\bZ/k$-manifold is the data of a compact oriented $n$-manifold $\overline{W}$, a closed oriented $(n-1)$-manifold $\beta W$, and an oriented identification $b: \partial \overline{W} \overset{\sim}\to \beta W \times \bZ/k$. We then write $W$ for the space obtained by identifying the $k$ copies of $\beta W$, and call $\overline{W}$ its resolution. This may be done in the category of smooth, PL, or topological manfolds. In the smooth case the data $(\overline{W}, \beta W, b, f : W \to X)$ represents a class in $MSO_n(X;\bZ/k)$. The notion of cobordism of $\bZ/k$-manifolds, and of maps out of them, is evident.

A $4n$-dimensional oriented manifold with boundary $(V, \partial V)$ still has a signature $\sigma(V) \in \bZ$ defined algebraically as the signature of the (possibly degenerate) symmetric form
$$H^{2n}(V, \partial V;\bR) \otimes H^{2n}(V, \partial V;\bR) \lra H^{4n}(V, \partial V;\bR) \overset{- \frown [V, \partial V]}\lra \bR.$$
The invariant $\sigma(W) := \sigma(\overline{W}) \mod k$ of a $\bZ/k$-manifold $W$ is a cobordism invariant \cite[Proposition 1.3]{MorganSullivan}, as a consequence of Novikov's additivity theorem for the signature and the usual cobordism invariance of the signature.

\subsection{The Sullivan orientation}
Let $MSPL_n$ be the $n$th space in the oriented $PL$ cobordism spectrum, i.e.\ the Thom space of the universal bundle over $BSPL(n)$. Then by PL transversality, $MSO_{n+4i}(MSPL_n)$ may be interpreted as cobordism classes of pairs $(M^{4i+n} \supset W^{4i})$ of an oriented smooth manifold and an oriented $PL$-submanifold, and assigning to this the signature $\sigma(W^{4i})$ gives a map
$$\phi_\bQ :  MSO_{4*}(\MSPL;\bQ) = \colim_{n \to \infty} MSO_{n+4*}(MSPL_n ; \bQ) \lra \bQ.$$
The $MSO_*$-module structure is given by $[X] \cdot [M^{4i+n} \subset W^{4i}] = [X \times M^{4i+n} \subset X \times W^{4i}]$, so $\phi_\bQ$ satisfies \eqref{eq:Product}. Similarly, $MSO_{n+4i}(MSPL_n ; \bZ/k)$ may be interpreted as cobordism classes of pairs $(M^{4i+n} \supset W^{4i})$ of a smooth $\bZ/k$-manifold and a $PL$ $\bZ/k$-submanifold, and assigning to this the signature $\sigma(W^{4i})$ as defined above gives a map
$$\phi_k :  MSO_{4*}(\MSPL;\bZ/k) = \colim_{n \to \infty} MSO_{n+4*}(MSPL_n ; \bZ/k) \lra \bZ/k.$$
These maps are compatible, and determine a homotopy class of maps of spectra
$$\Delta_{PL} : \MSPL \lra \KO[\tfrac{1}{2}]$$
(\emph{a priori} only well-defined up to phantom maps, as $\MSPL$ is not finite, but in fact unique \cite[\S 5.D]{MadsenMilgram}). By the same discussion with $\MSPL$ replaced by $\MSO$, $\Delta_{PL}$ restricts to $\Delta$. Furthermore, as the fibre of $BPL \to BTop$ is a $K(\bZ/2,3)$ \cite[p.\ 246]{KS}, it follows that the fibre of $\MSPL \to \MSTop$ is 2-local, and so $\Delta_{PL}$ canonically extends to a $\Delta_{Top} : \MSTop \to \KO[\tfrac{1}{2}]$. (Alternatively one can repeat the construction using topological transversality, but we can avoid this for now.)

\subsection{Twisted signatures}\label{sec:TwistedSignaturesKO}
Returning to a nondegenerate $(-1)^n$-symmetric bilinear form $(H_\bR, \lambda)$, if $f : W^{4i-2n} \to B\Aut(H_\bR, \lambda)$ is a map from a smooth $\bZ/k$-manifold then there is an associated symmetric bilinear form
$$H^{2i-n}(\overline{W}, \partial \overline{W} ; f^*\mathcal{H}_\bR) \otimes H^{2i-n}(\overline{W}, \partial \overline{W} ; f^*\mathcal{H}_\bR) \lra \bR$$
given by cup product, applying $\lambda$, then capping with the fundamental class, whose signature we call $\sigma(W; f)$. This number taken modulo $k$ is again $\bZ/k$-cobordism invariant, replacing Novikov additivity and cobordism invariance of the signature with the analogues \cite[Sätze I.3.1, I.3.2]{Meyer} for twisted signatures proved by Meyer. The assignment $[W, f] \mapsto \sigma(W;f)$ defines a map
$$\sign_k : MSO_{4i-2n}(B\Aut(H_\bR, \lambda);\bZ/k) \lra \bZ/k,$$
and these are easily checked to be compatible with each other under divisibility, and compatible with the analogue $\sign_\bQ = \sigma : MSO_{4i-2n}(B\Aut(H_\bR, \lambda)) \to \bZ$ for closed manifolds. 

Furthermore, if $N^{4j}$ is a smooth oriented closed manifold then we can form the map $W \times N \overset{\pi_W}\to W \overset{f}\to B\Aut(H_\bR, \lambda)$. By the K{\"u}nneth theorem we have
$$H^{2j+2i-n}(\overline{W} \times N, \partial \overline{W} \times N ; \pi_W^* f^* \mathcal{H}_\bR) \cong \bigoplus_{a+b = 2j+2i-n} H^a(\overline{W}, \partial \overline{W} ; f^*\mathcal{H}_\bR) \otimes H^b(N ; \bR),$$
and the sum $L$ of those terms with $b>2j$ is a sublagrangian with $L^\perp$ given by the sum of those terms with $b \geq 2j$ along with the radical, and so with
$$L^\perp/L \cong H^{2i-n}(\overline{W}, \partial \overline{W}  ; f^*\mathcal{H}_\bR) \otimes H^{2j}(N ; \bR) + \text{radical}.$$
We therefore have $\sigma(W \times N; f \circ \pi_W) = \sigma(W; f) \cdot \sigma(N)$. Similarly for closed manifolds.

By the discussion in Section \ref{sec:CobCoeffKthy}, this data corresponds to a map
$$\sign : B\Aut(H_\bR, \lambda) \lra \Omega^{\infty+2n} \KO[\tfrac{1}{2}],$$
well-defined up to homotopy and phantom maps.

\subsection{The family signature theorem over $\KO[\tfrac{1}{2}]$}

As we have already mentioned in Section \ref{sec:FSTQ}, a topological block bundle $\pi: E \to |K|$ with $d$-dimensional fibres has a stable vertical tangent microbundle $T^s_\pi E$ of virtual dimension $d$, and so a stable vertical normal microbundlebundle $\nu_\pi$ of virtual dimension $-d$, given by \cite[Section 2.4]{HLLRW}. If the block bundle is oriented, then these two microbundles obtain orientations. Using the constructions in that section\footnote{Specifically, the discussion there shows that we may find an exhaustion $|K|^0 \subset |K|^1 \subset |K|^2 \subset \cdots$ of $|K|$, embeddings $E\vert_{|K|^n} \subset |K|^n \times \bR^n$, open neighbourhoods $U^n$ of these with homeomorphisms $U^n \cong \nu_{n-d}$ to $\bR^{n-d}$-bundles. Then there are collapse maps $|K|^n_+ \wedge S^n \to (U^n)^+ \cong \mathrm{Th}(\nu_{n-d} \to E\vert_{|K|^n})$ with adjoints $|K|^n \to \Omega^n \mathrm{Th}(\nu_{n-d} \to E\vert_{|K|^n})$. Furthermore, the discussion shows that all this data can be chosen compatibly in $n$, so these maps assemble to the required Gysin map.} there is a Pontrjagin--Thom, or Gysin, map $\pi_! : |K| \to \Omega^\infty \mathrm{Th}(\nu_\pi \to E)$ and hence, Thomifying the map $E \to BSTop$ classifying $\nu_\pi$, a map
$$\alpha : |K| \overset{\pi_!}\lra \Omega^\infty \mathrm{Th}(\nu_\pi \to E) \lra \Omega^{\infty+d}\MSTop.$$

\begin{theorem}[Family signature theorem over ${\KO[\tfrac{1}{2}]}$]\label{thm:FSTSullivan}
Let $\pi: E \to |K|$ be an oriented topological block bundle with fibre $F^d$. If $d=2n$ let $\phi : |K| \to B\Aut(H^n(F;\bR), \lambda)$ classify the local coefficient system $\mathcal{H}^{n}(F;\bR)$ over $|K|$ with the $(-1)^n$-symmetric bilinear form given by cup product. Then the square
\begin{equation*}
\begin{tikzcd}[ampersand replacement=\&]
{|K|} \rar{\alpha} \arrow[d, "\phi"] \& \Omega^{\infty+d} \MSTop \arrow[d, "\Omega^{\infty+d}\Delta_{Top}"]\\
 {\begin{cases}
B\Aut(H^n(F;\bR), \lambda) & d=2n \\
* & d=2n+1
\end{cases}} \rar{\sign} \& {\Omega^{\infty+d} \KO[\tfrac{1}{2}]}
\end{tikzcd}
\end{equation*}
commutes up to homotopy and phantom maps.
\end{theorem}
\begin{proof}
The two ways around the square give two elements of $KO^d(|K| ; \bZ[\tfrac{1}{2}])$, which we must compare. By the discussion in Section \ref{sec:CobCoeffKthy} each of these corresponds to compatible maps $MSO_{4i-d}(|K|) \to \bZ[\tfrac{1}{2}]$ and $MSO_{4i-d}(|K| ; \bZ/k) \to \bZ/k$, and so we shall compare these.

That the maps $MSO_{4i-d}(|K|) \to \bZ[\tfrac{1}{2}]$ obtained by going the two ways around the square agree follows from Theorem \ref{thm:FSTQ} (it is the same as saying that the square commutes after rationalising ${\Omega^{\infty+d} KO[\tfrac{1}{2}]}$, whereupon the top composition is the collection of the $\int_\pi \cL_i(T_\pi^s E)$ and the bottom composition is the collection of the $\phi^*\sigma_{4i-d}$), so it remains to compare the two maps $MSO_{4i-d}(|K|;\bZ/k) \to \bZ/k$.

Let $f : W^{4i-d} \to |K|$ be a map from a $\bZ/k$-manifold; we may homotope it to be simplicial with respect to some triangulation of $W$, giving a topological block bundle $ f^* E \to  W$, whose total space $f^*E$ is a topological $\bZ/k$-manifold. We would like to say that the top composition $\Omega^{\infty+d}\Delta_{Top} \circ \alpha$ assigns to $(W, f)$ the signature $\sigma(f^*E)$ of the topological $\bZ/k$-manifold $f^* E$, which we have defined to be the signature of the (possibly degenerate) intersection form on the resolution $\overline{f^* E}$ , taken modulo $k$. This is true, but as the map $\Delta_{Top}$ was obtained by obstruction theory from the more meaningful map $\Delta_{PL} : \MSPL \to \KO[\tfrac{1}{2}]$, it requires a small argument. Namely, as $\MSPL \to \MSTop$ is a $\bZ[\tfrac{1}{2}]$-equivalence there is an $N \gg 1$ such that the element $2^N \alpha_*(W, f)$ lifts to $\Omega^{\infty+d}\MSPL$. In other words, the disjoint union of $2^N$ copies of the topological $\bZ/k$-manifold $f^* E$ is topologically cobordant to a $PL$ $\bZ/k$-manifold $E'$, and $\Delta_{Top}$ assigns to this $\sigma(E') = 2^N \sigma(f^*E)$: as we work with $2$ inverted, $\Delta_{Top}$ assigns $\sigma(f^*E)$ to $\alpha_*(W, f)$. (If $\Delta_{Top}$ is defined using topological transversality then we can of course omit this step.)

The block bundle $\overline{f^* E} \to \overline{W}$ has a relative Serre spectral sequence
$$E_2^{p,q} = H^p(\overline{W}, \partial \overline{W} ; \mathcal{H}^q(F;\bR)) \Longrightarrow H^{p+q}(\overline{f^* E}, \partial \overline{f^* E} ; \bR).$$
In parallel to the discussion in Section \ref{sec:PoincareFib}, by \cite[Satz I.1.5]{Meyer} the signature of the form on $E_2^{*,*}$ is the same as that of $E_\infty^{*,*}$, and the latter is the same as the signature of $\overline{f^* E}$. Furthermore $L := \bigoplus_{q > (4i-d)/2} E_2^{*, q}$ is again a sublagrangian of $E_2^{*,*}$, and as in Section \ref{sec:TwistedSignaturesKO} we have $L^\perp = \bigoplus_{q \geq (4i-d)/2} E_2^{*, q} + \text{radical}$ and so
$$L^\perp/L \cong H^{(4i-d)/2}(\overline{W}, \partial \overline{W} ; \mathcal{H}^{d/2}(F;\bR)) + \text{radical}.$$
In particular, if $d$ is odd then this is radical so has signature 0, and if $d=2n$ then this has signature $\sign_k(W, \phi)$. This is tautologically what the composition $\sign \circ \phi$ assigns to $(W, f)$ too, as required.
\end{proof}

\subsection{A formula for twisted signatures}\label{sec:TwistedFormulaKO}

For applications we also want to know how to evaluate the map $\sign : B\Aut(H_\bR, \lambda) \to \Omega^{\infty+2n} \KO[\tfrac{1}{2}]$ in concrete terms: in other words, to have an analogue of Meyer's formula \eqref{eq:MayerFormula}. Such an analogue is as follows, where $\xi \in K^0(B\Aut(H_\bR, \lambda))$ is the class constructed in Section \ref{sec:MeyerFormula} which appears in Meyer's formula.

\begin{theorem}\label{thm:KOSigFormula}
The square
\begin{equation*}
\begin{tikzcd}
B\Aut(H_\bR, \lambda) \rar{\sign} \dar{\psi^2\xi} & \Omega^{\infty+2n} \KO[\tfrac{1}{2}] \\
\Omega^{\infty} \K \rar{b^{n}}& \Omega^{\infty+2n} \K \uar{\tfrac{1}{2} r}
\end{tikzcd}
\end{equation*}
commutes up to homotopy and phantom maps.
\end{theorem}
\begin{proof}
We first argue that the map $\sign$ factors canonically (up to phantom maps) as
$$\sign : B\Aut(H_\bR, \lambda) \lra B\Aut(H_\bR, \lambda)^{top} \overset{\sign'}\lra \Omega^{\infty+2n} \KO[\tfrac{1}{2}],$$
 through the classifying space of the topologised variant
$$\Aut(H_\bR, \lambda)^{top} \cong \begin{cases}
Sp_{2g}(\bR) \simeq U(g) & n \text{ odd}\\
O_{p,q}(\bR) \simeq O(p) \times O(q) & n \text{ even}.
\end{cases}$$
This is certainly necessary, since the lower composition factors over this space, by replacing $\xi$ by the construction of Section \ref{sec:MeyerFormula} applied to the non-flat universal bundle $\mathcal{H}_\bR$ over $B\Aut(H_\bR, \lambda)^{top}$.

The map $\sign_\bQ$ can be interpreted as assigning to a manifold $W^{4i-2n}$ and a flat vector bundle $V := f^* \mathcal{H}_\bR \to W$ with a nondegenerate $(-1)^n$-symmetric bilinear form the index of the operator constructed by Meyer in the proof of \cite[Satz II.4.1]{Meyer}, but the definition of this operator does not require $V$ to be flat so $\sign_\bQ$ factors through a $\sign'_\bQ : MSO_{4i-2n}(B\Aut(H_\bR, \lambda)^{top}) \to \bZ$. Similarly, making use of index theory for $\bZ/k$-manifolds (as in \cite{FreedMelrose}, see also \cite{Rosenberg}) we recognise $\sign_k(W; f)$ as the index of the signature operator on the manifold with boundary $\overline{W}$ twisted, in the same way as by Meyer, by the flat vector bundle $f^* \mathcal{H}_\bR$ with its $(-1)^n$-symmetric structure. This operator is considered with the same Atiyah--Patodi--Singer boundary conditions on each of the $k$ copies of $\beta W$ forming its boundary, and having done so its index is well defined modulo $k$. As above, forming this index does not use the flatness of the vector bundle, so $\sign_k$ factors through a $\sign'_k : MSO_{4i-2n}(B\Aut(H_\bR, \lambda)^{top};\bZ/k) \to \bZ/k$, in total giving the claimed factorisation.

\vspace{1ex}
\noindent\textbf{Claim.} Maps
$$\begin{cases}
BU(g) & n \text{ odd}\\
BO(p) \times BO(q) & n \text{ even}
\end{cases} \lra \Omega^{\infty+2n} \KO[\tfrac{1}{2}]$$
are determined by their rationalisations, i.e.\ by their Pontrjagin character. 

\begin{proof}[Proof of Claim]
This may be deduced from the Atiyah--Segal completion theorem as follows (see \cite[Theorem 4.29]{MadsenMilgram} for a similar argument). We first reduce to the same statement for $\K[\tfrac{1}{2}]$ and the Chern character, as $\KO[\tfrac{1}{2}]$ is a retract of $\K[\tfrac{1}{2}]$, and then by Bott periodicity we can remove the $\Omega^{2n}$. Writing $G$ for $U(g)$ or $O(p) \times O(q)$, $R(G)$ for its complex representation ring, $I(G)$ for the augmentation ideal, and $BG^{(n)}$ for a $n$-skeleton of $BG$, by \cite[Theorem 2.1]{AtiyahSegal} the maps
$$R(G)/I(G)^n \lra K^0(BG^{(n)})$$
induce an isomorphism of pro-rings, and $K^{-1}(BG^{(n)})$ is pro-zero (and hence Mittag--Leffler). Using Milnor's $\lim^1$-sequence in $\K[\tfrac{1}{2}]$-theory we therefore have
$$K^0(BG ; \bZ[\tfrac{1}{2}]) \cong \lim_n K^0(BG^{(n)} ; \bZ[\tfrac{1}{2}]) \cong \lim_n R(G) \otimes \bZ[\tfrac{1}{2}]/(I(G) \otimes \bZ[\tfrac{1}{2}])^n.$$
The claim now follows by direct calculation with a presentation of the representation rings involved (the fact that $O(p) \times O(q)$ is not connected is ameliorated by our working with 2 inverted).
\end{proof}

By the index theorem applied to Meyer's signature operator twisted by a non-necessarily flat vector bundle with a nondegenerate $(-1)^n$-symmetric bilinear form, we have $\mathrm{ph}(\sign') = \mathrm{ch}(\psi^2(\xi))$. At this point we invoke the real forms of $\xi$ discussed in Remark \ref{rem:RealXi}: if $n$ is even then $\xi = c(\xi_\bR)$ and if $n$ is odd then $\xi = b^{-1} \cdot c(\xi_\bR)$. In either case we have $b^n  \xi = c(a^{\lfloor n/2 \rfloor} \cdot \xi_\bR)$ and so using that $\psi^2(b) = 2b$ and that Adams operations commute with complexification and realification \cite[Proposition IV.7.40]{KaroubiBook} we have
\begin{align*}
\mathrm{ph}(r(b^n \psi^2\xi)) &= 2^{-n} \mathrm{ph}(\psi^2 r(b^n \xi))\\
&= 2^{-n}\mathrm{ch}(c(\psi^2 r(c(a^{\lfloor n/2 \rfloor} \cdot\xi_\bR))))\\
&= 2 \cdot 2^{-n} \mathrm{ch}(c(\psi^2 (a^{\lfloor n/2 \rfloor} \cdot \xi_\bR)))\\
&= 2 \cdot 2^{-n} \mathrm{ch}(\psi^2 (b^{2\lfloor n/2 \rfloor} c(\xi_\bR)))\\
&= 2 \cdot 2^{-n} \mathrm{ch}(\psi^2 (b^n \xi))\\
&= 2  \mathrm{ch}(b^n \psi^2 ( \xi))\\
 &= 2  \mathrm{ch}(\psi^2 ( \xi)) 
\end{align*}
and so $\mathrm{ch}(\psi^2 ( \xi)) = \mathrm{ph}(\tfrac{1}{2} r(b^n \psi^2\xi))$. Thus $\mathrm{ph}(\sign') = \mathrm{ph}(\tfrac{1}{2} r(b^n \psi^2\xi))$ and so $\sign' \simeq \tfrac{1}{2} r(b^n \psi^2\xi)$.
\end{proof}

\section{Ranicki}

We now explain the strongest formulation of the family signature theorem for topological block bundles, as an equation in the generalised cohomology theory given by the symmetric $L$-theory of the integers, using Ranicki's ideas on $L$-theory and algebraic surgery (a vast literature, but in particular \cite{RanickiBook}). Ranicki's work in this direction was visionary, but his specific technical implementation is not ideal for our purposes (see Remark \ref{rem:RanickiBookIssues}). Instead, we will use the recent framework of Calm{\`e}s, Dotto, Harpaz, Hebestreit, Land, Moi, Nardin, Nikolaus, and Steimle \cite{No9I, No9II, No9III}. 

\subsection{The family signature theorem over $\LL^s(\bZ)$}

Let $\pi: E \to |K|$ be an oriented topological block bundle with fibre $F^d$. If $d=2n$ then $H^n(F;\bZ)/tors$ is equipped with its intersection form $\lambda$, which is nondegenerate and $(-1)^n$-symmetric, and the monodromy of this family gives a map
$$\phi: |K| \lra B\mathrm{Aut}(H^n(F;\bZ)/tors, \lambda).$$
If $d=2n+1$ then $tors\, H^{n+1}(F;\bZ)$ is equipped with its linking form $\ell$, which is nondegenerate and $(-1)^{n+1}$-symmetric, and the monodromy of this family gives a map
$$\phi : |K| \lra B\mathrm{Aut}(tors\, H^{n+1}(F;\bZ), \ell).$$

Writing $\LL^s(\bZ)$ for the symmetric $L$-theory spectrum of the integers, and writing $\sigma : \MSTop \to \LL^s(\bZ)$ for the Ranicki orientation (as in \cite[Proposition 15.8]{RanickiPreprint}, \cite[p.\ 287]{RanickiTSO}, \cite[Proposition 16.1]{RanickiBook}, \cite[Proposition 7.10]{KMM}, \cite{LauresMcClure}; we will give our own definition in Section \ref{sec:RanickiOrientation}), our main result is as follows.

\begin{theorem}[Family signature theorem over $\LL^s(\bZ)$]\label{thm:FSTLThy}
Let $\pi: E \to |K|$ be an oriented topological block bundle with fibre $F^d$. Then the square
\begin{equation*}
\begin{tikzcd}[ampersand replacement=\&]
{|K|} \rar{\alpha} \arrow[d, "\phi"] \& \Omega^{\infty+d} \MSTop \arrow[d, "\Omega^{\infty+d}\sigma"]\\
 {\begin{cases}
B\mathrm{Aut}(H^n(F;\bZ)/tors, \lambda) & d=2n \\
B\mathrm{Aut}(tors\, H^{n+1}(F;\bZ), \ell) & d=2n+1
\end{cases}} \rar{inc} \& {\Omega^{\infty+d} \LL^s(\bZ)}
\end{tikzcd}
\end{equation*}
commutes up to homotopy.
\end{theorem}

The bottom map $inc$ in this diagram arises from considering isomorphisms of (linking) forms as being cobordisms. We will construct this map in Section \ref{sec:MapInc}, and construct the Ranicki orientation in Section \ref{sec:RanickiOrientation}. Our proof of this theorem uses only formal properties of symmetric $L$-theory. In Section \ref{sec:Addenda} we will explain how the result can be somewhat improved, and interpreted, using non-formal results about the homotopy type of $\LL^s(\bZ)$ and its relation with Grothendieck--Witt theory.

We will obtain Theorem \ref{thm:FSTLThy} as the combination of two results, each of which holds for a class of families $\pi : E \to |K|$ more general than oriented topological block bundles. Firstly, when $\pi : E \to |K|$ is a Poincar{\'e} mock bundle (defined in Section \ref{sec:PMock}) we will define a family signature map $|K| \to \Omega^{\infty+d} \LL^s(\bZ)$. Secondly, when $\pi$ is in fact a manifold mock bundle (defined in Section \ref{sec:Mock}) we shall explain why the family signature map factors as $\Omega^{\infty+d} \sigma \circ \alpha : |K| \to \Omega^{\infty+d} \MSTop \to \Omega^{\infty+d} \LL^s(\bZ)$. Thirdly, when the Poincar{\'e} mock bundle $\pi$ comes from a fibration with Poincar{\'e} fibre (and in fact even something slightly more general) we will prove that the family signature map factors as $inc \circ \phi$. Hence when all these conditions hold, such as for topological block bundles, we obtain the conclusion of Theorem \ref{thm:FSTLThy}.

\subsection{Manifold mock bundles}\label{sec:Mock}
We shall describe a model for $\Omega^{\infty+d} \MSTop$ as the classifying space for (oriented, topological, codimension $-d$) \emph{mock bundles} in the sense of \cite[\S II]{BRS}. The construction is similar to the ``ad theories'' of \cite{QuinnAd, LauresMcClure}, but we spell things out explicitly.

Write $\Delta^{p-1}_i \subset \Delta^p$ for the $(p-1)$-simplex spanned by all vertices except the $i$th, and $\Delta^p_i(\epsilon) := \{(t_0, t_1, \ldots, t_p) \in \Delta^p \, | \, 0 \leq t_i < \epsilon\}$, with the projection map
\begin{align*}
\pi_i(\epsilon) : \Delta^p_i(\epsilon) &\lra \Delta^{p-1}_i\\
(t_0, t_1, \ldots, t_p) &\longmapsto (\tfrac{t_0}{1-t_i}, \tfrac{t_1}{1-t_i}, \ldots, \tfrac{t_{i-1}}{1-t_i}, \tfrac{t_{i+1}}{1-t_i}, \ldots, \tfrac{t_{p}}{1-t_i}).
\end{align*}
The following is parallel to \cite[Definition 2.3.1]{HLLRW}.

\begin{definition}\label{defn:BordismSpace}
For $d \in \bZ$, let $\Mock(d,n)$ be the semi-simplicial set with $p$-simplices given by locally flat compact topological $(d+p)$-dimensional oriented submanifolds $E \subset \Delta^p \times \bR^n$ such that for each $i=0,1,\ldots, p$, 
\begin{enumerate}[(i)]
\item $E$ is topologically transverse to each $\Delta^{p-1}_i \times \bR^n$,

\item there is an $\epsilon>0$ such that
$$E \cap (\Delta^p_i(\epsilon) \times \bR^n) = (\pi_i(\epsilon) \times \bR^n)^{-1}(E \cap (\Delta^{p-1}_i \times \bR^n)),$$

\item $\partial E = E \cap (\partial \Delta^p \times \bR^n)$.
\end{enumerate}
Define face maps $d_i : \Mock(d,n)_p \to \Mock(d,n)_{p-1}$ by restricting $E$ to the $i$th face $\Delta_i^{p-1} \subset \Delta^p$, and giving it the induced orientation. Let $\Mock(d) := \colim\limits_{n \to \infty} \Mock(d,n)$.
\end{definition}

This semi-simplicial set is Kan, by the same discussion as the paragraph after \cite[Definition 2.3.1]{HLLRW}. The space $|\Mock(d)|$ carries a tautological family of manifolds. Let
$${E}(d)_p := \{(E; t_0, t_1, \ldots, t_p ; x) \in \Mock(d)_p \times \Delta^p \times \bR^\infty \, | \, (t_0, t_1, \ldots, t_p ; x) \in E\},$$
and $\pi(d)_p : {E}(d)_p \to \Mock(d)_p \times \Delta^p$ denote the projection map. The maps $\pi(d)_p$ assemble to a map
$$\pi(d) : {E}(d) \lra |\Mock(d)|$$
where ${E}(d) := \left(\bigsqcup_{p \geq 0} {E}(d)_p\right){/\sim}$ with
$$(E; t_0, t_1, \ldots, t_{i-1}, 0, t_{i+1}, \ldots t_p ; x) \sim (d_i(E); t_0, t_1, \ldots, t_{i-1}, t_{i+1}, \ldots t_p ; x).$$
This family is tautological in the sense that $\pi(d)^{-1}(\{E\} \times \Delta^p) = \{E\} \times E$. If $\phi: K \to \Mock(d)$ is a semi-simplicial map then we can form the pullback 
$$\phi^*\pi(d) :  \phi^*{E}(d) \lra |K|;$$
this is an (oriented, topological, codimension $-d$) mock bundle over $|K|$.

We wish to explain why the spaces $|\Mock(d)|$ arise as the $d$th spaces of an $\Omega$-spectrum.

\begin{definition}
Let $\Mock_\partial(d+1, n)$ have $p$-simplices given by a $(d+1+p)$-manifold $E^{d+1+p} \subset \Delta^p \times \bR^{n} \times [0,\infty)$ satisfying (i) and (ii) from Definition \ref{defn:BordismSpace} as well as
\begin{enumerate}[(i')]
\setcounter{enumi}{2}
\item $E$ is topologically transverse to $\Delta^p \times \bR^{n} \times \{0\}$, and writing $\partial_v E := E \cap \Delta^p \times \{0\}$ there is a $\delta>0$ such that $E \cap (\Delta^p  \times \bR^{n} \times [0,\delta)) = \partial_v E \times [0,\delta)$, and $\partial E = (E \cap (\partial \Delta^p \times \bR^{n+1})) \cup \partial_v E$. 
\end{enumerate}
Let $\Mock_\partial(d+1) := \colim\limits_{n \to \infty} \Mock_\partial(d+1, n)$. 
\end{definition}

As above this semi-simplicial set is Kan, and carries a tautological family ${E}_\partial(d+1) \to |\Mock_\partial(d+1)|$ of $(d+1)$-manifolds with boundary, whose boundaries assemble into a family $\partial_v {E}_\partial(d+1) \to |\Mock_\partial(d+1)|$ of $d$-manifolds. Intersecting with $\Delta^p \times \bR^{n} \times \{0\}$ gives a semi-simplicial map
\begin{equation}\label{eq:ResMap}
res: \Mock_\partial(d+1) \lra \Mock(d),
\end{equation}
which is a Kan fibration, and classifies the family $\partial_v {E}_\partial(d+1)$. The fibre of this map over the empty manifold is isomorphic to $\Mock(d+1)$ (via a choice of homeomorphism $(0,\infty) \cong \bR$). One verifies that the connecting map $\partial : \pi_{i}(\Mock(d)) \to \pi_{i-1}(\Mock(d+1))$ is an isomorphism, as both sides are identified with cobordism classes of oriented topological $(d+i)$-manifolds and this map corresponds to the identity map. Thus there are equivalences
$$\Omega |\Mock(d)| \overset{\sim}\lra |\Mock(d+1)|$$
(and $|\Mock_\partial(d+1)| \simeq *$), and so an $\Omega$-spectrum $\Mock$ with $\Omega^{\infty+d} \Mock \simeq |\Mock(d)|$.

\subsection{Topological transversality}

There is a variant $\Mock'(d,n)$ of $\Mock(d,n)$ in which simplices $E \subset \Delta^p \times \bR^n$ are equipped with a choice of tubular neighbourhood (i.e.\ a $\bR^{n-d}$-bundle over $E$ with a homeomorphism to an open neighbourhood of $E$, both compatible with the collar structure given by (ii)) and the forgetful and Pontrjagin--Thom collapse maps give a zig-zag
$$|\Mock(d,n)| \longleftarrow |\Mock'(d,n)| \lra \Omega^n \mathrm{Th}(\gamma_{n-d} \to BSTop(n-d)).$$
In the limit as $n\to \infty$ the leftwards map is an equivalence (by the stable existence and uniqueness of normal microbundles \cite[p.\ 204]{KS} and the Kister--Mazur theorem \cite[p.\ 159]{KS}), and the right-hand side is the space $\Omega^{\infty+d} \mathrm{MSTop}$, giving a homotopy class of map
$$ |\Mock(d)| \lra \Omega^{\infty+d} \mathrm{MSTop}.$$

It follows from topological transversality (\cite[p.\ 85]{KS}, \cite[\S 9.6]{FreedmanQuinn}) that this map is a homotopy equivalence. Furthermore this discussion identifies the spectrum $\Mock$ with $\MSTop$, by comparing the analogue of the map \eqref{eq:ResMap} for manifolds in $\bR^{n} \times [0,\infty)$ with the path-fibration over $\Omega^n \mathrm{Th}(\gamma_{n-d} \to BSTop(n-d))$.

\subsection{Poincar{\'e} mock bundles}\label{sec:PMock}

If $\tau$ is a simplex, $E \subset \tau \times \bR^\infty$ is a closed subset, and $\sigma$ is a face of $\tau$, write $E_\sigma := (\sigma \times \bR^\infty) \cap E$. Write $E_{\partial \tau} = \cup_{\sigma < \tau} E_\sigma$, the union over all proper faces.

\begin{definition}\label{defn:PoincareBordismSpace}
For $d \in \bZ$, let $\Mock^P(d)$ be the semi-simplicial set with $p$-simplices given by pairs of a closed subset $E \subset \Delta^p \times \bR^\infty$ and a singular chain $[E_\sigma] \in C_{d+|\sigma|}(E_\sigma;\bZ)$ for each face $\sigma$ of $\Delta^p$, such that 
\begin{enumerate}[(i)]
\item each $E_\sigma$ is homotopy equivalent to a finite CW-complex,

\item $d[E_\sigma] = \sum_{i=0}^{|\sigma|} (-1)^i [E_{d_i \sigma}] \in C_{d+|\sigma|-1}(E_\sigma;\bZ)$, so that in particular $[E_\sigma]$ is a cycle in $C_{d+|\sigma|}(E_\sigma, E_{\partial \sigma};\bZ)$,

\item the maps 
\begin{align*}
- \frown [E_\sigma] : H^*(E_\sigma ; \bZ) &\lra H_{d+|\sigma|-*}(E_\sigma, E_{\partial \sigma}; \bZ)\\
- \frown d[E_\sigma] : H^*(E_{\partial \sigma} ; \bZ) &\lra H_{d-1+|\sigma|-*}(E_{\partial \sigma}; \bZ)
\end{align*}
are isomorphisms.
\end{enumerate}
Define face maps $d_i : \Mock^P(d)_p \to \Mock^P(d)_{p-1}$ by restricting $E$ to the $i$th face $\Delta_i^{p-1} \subset \Delta^p$, and taking those $[E_\sigma]$ with $\sigma \leq \Delta_i^{p-1}$.
\end{definition}

\begin{remark}
It is worth emphasising that such $E$'s are not Poincar{\'e} complexes (or -ads) in the sense of \cite[Section 2]{WallBook}---and so in the sense we have used earlier in this paper---as we are only asking for duality with $\bZ$-coefficients rather than with all local coefficients. This is deliberate! 
\end{remark}

Similarly to the case of mock bundles, the semi-simplicial set $\Mock^P(d)$  is Kan (the evident analogues of \cite[Lemma 1.2, Theorem 2.1 Addendum]{WallPoincare} for duality with $\bZ$-coefficients are used for verifying this) and the space $|\Mock^P(d)|$ carries a tautological family, given by defining
$${E}^P(d)_p := \{(E; t_0, t_1, \ldots, t_p ; x) \in \Mock^P(d)_p \times \Delta^p \times \bR^\infty \, | \, (t_0, t_1, \ldots, t_p ; x) \in E\},$$
and letting $\pi(d)_p : {E}^P(d)_p \to \Mock^P(d)_p \times \Delta^p$ denote the projection maps, which assemble to $\pi(d) : {E}^P(d) \lra |\Mock^P(d)|$. This tautological family can be pulled back along a semi-simplicial map $\phi : K \to \Mock^P(d)$ to obtain a Poincar{\'e} mock bundle over $|K|$.

Just as in the last section, one introduces the evident analogue $\Mock^P_\partial(d+1)$ (see the proof of Lemma \ref{lem:RanickiOr} below for a definition), to obtain equivalences $\Omega |\Mock^P(d)| \overset{\sim}\to |\Mock^P(d+1)|$ and so an associated $\Omega$-spectrum $\Mock^P$. 

\begin{remark}
The spectrum $\Mock^P$ is equivalent to Levitt's \cite{LevittPoincare} Poincar{\'e} bordism spectrum. In particular, there is a Pontrjagin--Thom map $\Mock^P \to \mathrm{MSG}$ but it is not an equivalence (its fibre is the 0-connected cover of quadratic $L$-theory, cf.\ \cite[Remark 19.9]{RanickiBook}, explicated in \cite[Remark 2.3]{MarkusNote}). The comparison with Levitt's theory uses the fact that the present naive definition of Poincar{\'e} complexes still have Spivak normal fibrations \cite{BrowderSpivak}, which is Levitt's definition of a Poincar{\'e} complex.
\end{remark}

Compact oriented topological manifolds with boundary admit fundamental chains with respect to which they have Poincar{\'e}--Lefschetz duality. We may choose such chains by induction over the skeleta of $\Mock(d)$ as follows. For each $0$-simplex $E \in \Mock(d)_0$ we choose a fundamental cycle for the oriented $d$-manifold $E$. Supposing such chains have been chosen for all simplices of $\Mock(d)$ of dimension $< p$, then for each $p$-simplex $E \in \Mock(d)_p$ we have a fundamental chain
$$\sum_{i=0}^{p} (-1)^i [E_{d_i(\Delta^p)}] \in C_{d+p-1}(E;\bZ)$$
for $E_{\partial \Delta^p}$, which is trivial in homology as this is the boundary of $E$ so we can take $[E] \in C_{d+p}(E;\bZ)$ to be a chain with boundary the above, and which also restricts to a fundamental cycle on any closed components of $E$. This gives (using also that compact topological manifolds are homotopy equivalent to finite CW-complexes) a map $\Mock(d) \lra \Mock^P(d)$, 
unique up to homotopy because at each stage we have chosen a top-dimensional chain of $E_\sigma$ which is unique up to a boundary. These maps assemble into a map of spectra
$$\Mock \lra \Mock^P$$
and hence give a map $\MSTop \to \Mock^P$.

\subsection{Categorical preliminaries}\label{sec:catPrelim}
In Subsections \ref{sec:catPrelim}--\ref{sec:SurgeryBelowMid} we will make extensive use of the definitions and constructions of \cite{No9I, No9II, No9III}. We will give references for all results that we use, but the reader will need to be familiar with the general set-up of those papers, which we will not review.

For a semi-simplicial set $K$, we will often work in the stable $\infty$-category 
$$\mathsf{C} := \mathsf{Fun}( \mathsf{Simp}(K)^{op}, \mathcal{D}^p(\bZ))$$
of functors from the (opposite of the) poset of simplices of $K$ to the derived $\infty$-category $\mathcal{D}^p(\bZ)$ of perfect $\bZ$-modules. On $\mathcal{D}^p(\bZ)$ we have the symmetric hermitian structure $\Qoppa = \Qoppa^s$ given by $\Qoppa(X) := \Hom_\bZ(X \otimes X, \bZ)^{hC_2}$, whose underlying bilinear functor is $B(X, Y) = \Hom_{\bZ}(X \otimes Y, \bZ)$ which is perfect with associated duality $D(X) = \Hom_\bZ(X, \bZ)$. The discussion in \cite[Construction 6.3.1]{No9I} shows that $\Qoppa$ induces a hermitian functor $\Qoppa_K$ on the cotensoring $\mathsf{C}$, and writing $\underline{\bZ}$ for the constant functor we can express this as 
$$\Qoppa_K(X) = \mathrm{Hom}_\mathsf{C}(X \otimes X, \underline{\bZ})^{hC_2},$$
where the tensor product in $\mathsf{C}$ is formed objectwise. By \cite[Proposition 6.3.2]{No9I} its underlying bilinear functor is $B_K(X,Y) = \Hom_\mathsf{C}(X \otimes Y, \underline{\bZ})$, and as $((\mathsf{Simp}(K)^{op})_{\sigma /})^{op}$ is finite (it is the poset of simplices of a fixed simplex $\sigma$) this bilinear functor is nondegenerate with associated duality
\begin{equation}\label{eq:KDuality}
D_K(X)(\sigma) = \holim_{\tau \in \mathsf{Simp}(\sigma)} \Hom_\bZ(X(\tau), \bZ).
\end{equation}
It is elementary to check that this duality is perfect (alternatively, observe that this may be checked for each simplex $K$, and appeal to \cite[Proposition 6.6.1]{No9I} for the finite simplicial complex given by a single simplex), making $(\mathsf{C}, \Qoppa_K)$ a Poincar{\'e} category.

\subsection{The Ranicki orientation}\label{sec:RanickiOrientation}

The Ranicki orientation is a certain map of spectra $\sigma: \MSTop \to \LL^s(\bZ)$. Its construction was outlined in \cite[Proposition 15.8]{RanickiPreprint} (for PL, rather than topological, manifolds), but the first rigorous implementation seems to not have been until \cite{LauresMcClure}, where it is constructed as a map of $E_1$-rings. We will give our own construction of this map, intuitively the same as Ranicki's but using the definition of the symmetric $L$-theory spectrum from \cite[Section 4.4]{No9II}. In fact we will construct a map $\sigma^P : \Mock^P \to \LL^s(\bZ)$, which may be precomposed with $\MSTop \to \Mock^P$ to obtain $\sigma$.

For any semi-simplicial set $K$ and semi-simplicial map $\phi : K \to \Mock^P(d)$ we can form the pullback 
$$\phi^*\pi := \phi^* {E}^P(d) \lra |K|,$$
a Poincar{\'e} mock bundle over $|K|$. For a simplex $\sigma$ of $K$ we write $E_\sigma := f(\sigma)$. This defines a functor
$$C : \mathsf{Simp}(K)^{op} \lra \mathcal{D}^p(\bZ)$$
by $C(\sigma) := C^*(E_\sigma ; \bZ)$, i.e.\ an object of the category $\mathsf{C}$ described above. 

We first explain how $C$ has the structure of a Poincar{\'e} object in $(\mathsf{C}, \Qoppa^{[-d]}_K)$.

\begin{lemma}\label{lem:FundChain}
There is a canonical morphism $[{E}] : C \to S^{-d} \otimes \underline{\bZ}$, a ``fibrewise fundamental chain''.
\end{lemma}
\begin{proof}
By adjunction a morphism as indicated is the same as a morphism
$$\hocolim_{\sigma \in \mathsf{Simp}(K)^{op}} C(\sigma) \lra S^{-d} \otimes \bZ$$
in $\mathcal{D}^p(\bZ)$, which is a morphism $S^d \to \holim\limits_{\sigma \in \mathsf{Simp}(K)^{op}} \mathrm{Hom}_\bZ(C(\sigma), \bZ)$. As $C(\sigma) = C^*(E_\sigma ; \bZ) = \mathrm{Hom}(C_*(E_\sigma;\bZ) , \bZ)$ and $C_*(E_\sigma;\bZ)$ is perfect, this is equivalent to a morphism
$$S^d \otimes \bZ \lra \holim_{\sigma \in \mathsf{Simp}(K)^{op}} C_*(E_\sigma;\bZ).$$

The data of the compatible fundamental chains $\{[E_\sigma]\}_{\sigma \in K}$ precisely gives such a $d$-cycle in this homotopy limit.
\end{proof}

Using this morphism we may form $q : C \otimes C \overset{-\smile -}\to C \overset{[{E}]}\to S^{-d} \otimes \underline{\bZ}$, which, as the multiplication on $C$ is commutative, determines a Hermitian form
$$q \in \Omega^\infty \Qoppa^{[-d]}_K(C) = \Omega^{\infty+d} \mathrm{Hom}_\mathsf{C}(C \otimes C, \underline{\bZ})^{hC_2}.$$
To see that this is nondegenerate, note that its adjoint $q_\sharp : C \to S^{-d} \otimes D(C)$ evaluated at $\sigma$ is the map
$$q_\sharp(\sigma) : C^*(E_\sigma;\bZ) \lra S^{-d} \otimes \holim_{\tau \subset \sigma}\mathrm{Hom}_\bZ(C^*(E_\tau;\bZ), \bZ)$$
given by cap product with the chain $[E_\sigma]$ and evaluation. There is an equivalence $C_*(E_\tau;\bZ) \overset{\sim}\to \mathrm{Hom}_\bZ(C^*(E_\tau;\bZ), \bZ)$ given by evaluation, which identifies this map with the map
$$q_\sharp(\sigma) : C^*(E_\sigma; \bZ) \lra S^{-d} \otimes \holim_{\tau \subset \sigma} C_*(E_\tau;\bZ) \simeq S^{-d-|\sigma|} \otimes C_*(E_\sigma, E_{\partial \sigma};\bZ)$$
given by cap product with the chain $[E_\sigma]$, and this is an equivalence by the definition of Poincar{\'e} mock bundle. Thus $(C, q)$ is indeed a Poincar{\'e} object in $(\mathsf{C}, \Qoppa^{[-d]}_K)$. 

Let us briefly recall how the $L$-theory space $\mathcal{L}(\bZ, \Qoppa^{[-d]})$ is defined in \cite[Section 4.4]{No9II}. It is the geometric realisation of the simplicial space with $p$-simplices given by the space(=$\infty$-groupoid) of Poincar{\'e} objects in the Poincar{\'e} category $\rho_p(\mathcal{D}^p(\bZ), \Qoppa^{[-d]}) := (\mathsf{Fun}(\mathsf{Simp}(\Delta^p)^{op}, \mathcal{D}^p(\bZ)), \Qoppa^{[-d]}_{\Delta^p})$.

We may then apply the following ``coassembly'' construction to $(C, q)$. For each $p$-simplex $\sigma : \Delta^p \to K$  we obtain a Poincar{\'e} object in $\rho_p(\mathcal{D}^p(\bZ), \Qoppa^{[-d]})$ by pullback: the underlying object is the functor
$$\sigma^*C : \mathsf{Simp}(\Delta^p)^{op} \lra \mathcal{D}^p(\bZ)$$
given by restriction of $C$, and it is made into a Poincar{\'e} object via the restriction $\sigma^*q$ of $q$. Applying this construction levelwise defines a map
$$\mathrm{Coass}(C, q) : |K| \lra |\mathrm{Pn} \rho(\mathcal{D}^p(\bZ), \Qoppa^{[-d]})| =: \mathcal{L}(\bZ, \Qoppa^{[-d]}).$$
We call this the \emph{family signature} of the Poincar{\'e} mock bundle $\pi: E \to |K|$. In particular, for the universal example $K=\Mock^P(d)$ with associated Poincar{\'e} object $(C_d, q_d)$ this construction defines a map
$$\Phi_d :  |\Mock^P(d)| \xrightarrow{\mathrm{Coass}(C_d, q_d)}  \mathcal{L}(\bZ, \Qoppa^{[-d]}).$$

The symmetric $L$-theory spectrum $\LL^s(\bZ)$ as described in the discussion after \cite[Corollary 4.4.5]{No9II} has $(-d)$th space equivalent to $\mathcal{L}(\bZ, \Qoppa^{[-d]})$ by \cite[Corollary 4.4.5]{No9II}, and the spectrum $\Mock^P$ has $(-d)$-th space $|\Mock^P(d)|$, so the maps $\Phi_{d}$ can in principle arise from a map of spectra. They do:

\begin{lemma}\label{lem:RanickiOr}
The $\Phi_d$ arise as $\Omega^{\infty+d}\sigma^P$ for a spectrum map $\sigma^P : \Mock^P \to \LL^s(\bZ)$. 
\end{lemma}

We define the composition
$$\sigma: \MSTop \simeq \Mock \lra \Mock^P \overset{\sigma^P}\lra \LL^s(\bZ)$$
to be the Ranicki orientation. The construction is conceptually the same as the map constructed in \cite{LauresMcClure}, though implemented differently. It is also conceptually the same as the maps constructed in \cite[Proposition 15.7]{RanickiPreprint}, \cite[p.\ 289]{RanickiTSO}, \cite[Proposition 16.1]{RanickiBook} and \cite[Proposition 7.10 (1)]{KMM}, though see Remark \ref{rem:RanickiBookIssues}.

\begin{proof}
We must establish a compatibility between the $\Phi_d$'s, and to do so we give a precise definition of $\Mock_\partial^P(d+1)$. For a simplex $\tau$, a face $\sigma \leq \tau$, and a closed subset $E \subset \tau \times \bR^\infty \times [0,\infty)$, write $\partial_v E_\sigma := (\sigma \times \bR^\infty \times \{0\}) \cap E$.  Let $\Mock_\partial^P(d+1)$ have $p$-simplices given by a closed subset $E \subset \Delta^p \times \bR^{\infty} \times [0,\infty)$ and a singular chain $[E_\sigma] \in C_{d+1+|\sigma|}(E_\sigma;\bZ)$ for each face $\sigma$ of $\Delta^p$, such that
\begin{enumerate}[(i')]
\item each $E_\sigma$ and $\partial_v E_\sigma$ is homotopy equivalent to a finite CW-complex,
\item the chain $[\partial_v E_\sigma] := d[E_\sigma] - \sum_{i=0}^{|\sigma|} [E_{d_i \sigma}] \in  C_{d+|\sigma|}(E_\sigma ; \bZ)$ lies in the subgroup $C_{d+|\sigma|}(\partial_v E_\sigma ; \bZ)$, so $[E_\sigma]$ is a cycle in $C_{d+1+|\sigma|}(E_\sigma, \partial_v E_\sigma \cup E_{\partial \sigma} ; \bZ)$,
\item the maps
\begin{align*}
- \frown [E_\sigma] : H^*(E_\sigma ; \bZ) &\lra H_{d+1+|\sigma|-*}(E_\sigma, \partial_v E_\sigma \cup E_{\partial \sigma}; \bZ)\\
- \frown d[E_\sigma] : H^*(\partial_v E_\sigma \cup E_{\partial \sigma} ; \bZ) &\lra H_{d+|\sigma|-*}(\partial_v E_\sigma \cup E_{\partial \sigma}; \bZ)
\end{align*}
are isomorphisms.
\end{enumerate}
Assigning to $(E, \{[E_\sigma]\})$ the data $(\partial_v E, \{[\partial_v E_\sigma]\})$, where the $[\partial_v E_\sigma]$ are given by (ii'), defines a semi-simplicial map
$$res: \Mock_\partial^P(d+1) \lra \Mock^P(d),$$
which is furthermore a Kan fibration. (The key calculation is to show that the maps $- \frown [\partial_v E_\sigma] : H^*(\partial_v E_\sigma ; \bZ) \to H_{d+|\sigma|-*}(\partial_v E_\sigma, \partial_v E_{\partial \sigma};\bZ)$, and $- \frown d[\partial_v E_\sigma] :H^*(\partial_v E_{\partial \sigma};\bZ) \to H_{d-1+|\sigma|-*}(\partial_v E_{\partial \sigma};\bZ)$ are isomorphisms, which follows by repeatedly using \cite[Lemma 1.2 and Theorem 2.1 Addendum]{WallPoincare}.) The space $|\Mock_\partial^P(d+1)|$ carries a tautological family ${E}^P_\partial(d+1)$ and the map $res$ classifies the associated family $\partial_v {E}^P_\partial(d+1)$ given by intersecting with $\bR^\infty \times \{0\}$. The fibre of the map $res$ over the basepoint (which is the empty space) is isomorphic to $\Mock^P(d+1)$ (via a choice of homeomorphism $(0,\infty) \cong \bR$). 

We wish to establish a commutative diagram
\begin{equation*}
\begin{tikzcd}
{|\Mock^P(d+1)|} \rar \dar{\mathrm{Coass}(C_{d+1})}& {|\Mock_\partial^P(d+1)|} \rar{|res|} \dar& {|\Mock^P(d)|} \dar{\mathrm{Coass}(C_d)}\\
\mathcal{L}(\bZ, \Qoppa^{[-d-1]}) \rar& \mathcal{L}(\mathrm{Met}(\mathcal{D}^p(\bZ), \Qoppa^{[-d]})) \rar& \mathcal{L}(\bZ, \Qoppa^{[-d]}),
\end{tikzcd}
\end{equation*}
in which the rows are homotopy fibre sequences. The lower sequence comes from the metabolic Poincaré--Verdier sequence \cite[Example 1.2.5]{No9II} and the fact that $\mathcal{L}(-)$ is Verdier-localising and bordism-invariant \cite[Theorem 4.4.2]{No9II}. The two middle spaces are contractible so this will give the desired compatibility.
 
 To obtain the middle vertical map in this diagram observe that $\mathrm{Coass}(C_d) \circ |res|$ is the construction from above applied with $K = \Mock_{\partial}^P(d+1)$ and $\phi=res$, the map that classifies the family $\partial_v {E}^P_\partial(d+1) \to |\Mock_\partial^P(d+1)|$. Recall that Poincar{\'e} objects in $\mathrm{Met}(\mathcal{D}^p(\bZ), \Qoppa^{[-d]})$ encode Poincar{\'e} objects in $(\mathcal{D}^p(\bZ), \Qoppa^{[-d]})$ equipped with a lagrangian, so we must equip the Poincar{\'e} object $(res)^*(C_d, q_d)$ associated to $\partial_v {E}^P_\partial(d+1) \to |\Mock_\partial^P(d+1)|=|K|$ with a lagrangian. 
 
 This of course comes from the fact that $\partial_v {E}^P_\partial(d+1)$ is canonically the boundary of the family  ${E}_\partial^P(d+1) \to |\Mock_\partial^P(d+1)|$. Define a morphism $f : L \to (res)^*C_d$ in $\mathsf{C}$ by the restriction map
 $$f : L(E) := C^*(E) \lra C^*(\partial_v E) = ((res)^*C_d)(E)$$
 and choose the nullhomotopy of $f^*q_d$ given by the nullhomotopy of the composition $L \overset{f}\to (res)^*C_d \overset{[\partial_v {E}^P_\partial(d+1)]}\to S^{-d} \otimes \underline{\bZ}$
 provided by a fibrewise fundamental chain $[{E}^P_\partial(d+1)]$ of the family ${E}^P_\partial(d+1) \to |\Mock^P_\partial(d+1)|$ extending $[\partial_v {E}^P_\partial(d+1)]$, constructed as in Lemma \ref{lem:FundChain}. This gives the middle vertical map, and as we chose a fundamental chain, on $(d+1)$-manifolds without boundary it restricts to a map homotopic to $\mathrm{Coass}(C_{d+1})$.
\end{proof}

\subsection{The map $inc$}\label{sec:MapInc}

Recall that $\mathcal{L}(\bZ, \Qoppa^{[-d]})$ is defined as the geometric realisation of the simplicial space $[p] \mapsto \mathrm{Pn} \rho_p(\mathcal{D}^p(\bZ), \Qoppa^{[-d]})$, whose space of 0-simplices is $\mathrm{Pn}(\mathcal{D}^p(\bZ), \Qoppa^{[-d]})$, the space(=$\infty$-groupoid) of Poincar{\'e} objects in the Poincar{\'e} category $(\mathcal{D}^p(\bZ), \Qoppa^{[-d]})$.

Suppose first that $(H, \lambda)$ is a free $\bZ$-module $H$ with a $(-1)^n$-symmetric bilinear form $\lambda : H \otimes H \to \bZ$ which is nondegenerate, i.e.\ the adjoint $\lambda^\mathrm{ad} : H \to \mathrm{Hom}_\bZ(H, \bZ)$ is an isomorphism. Then there is a Poincar{\'e} object in $(\mathcal{D}^p(\bZ), \Qoppa^{[-2n]})$ given by $H[-n]$ equipped with the symmetric form
$$q_\lambda : H[-n] \otimes H[-n] \overset{\lambda}\lra \bZ[-2n].$$
As $H[-n]$ is concentrated in a single degree, the space of automorphisms of $(H[-n], q_\lambda)$ in the $\infty$-groupoid $\mathrm{Pn}(\mathcal{D}^p(\bZ), \Qoppa^{[-2n]})$ is homotopy-discrete, and is equivalent to $\mathrm{Aut}(H, \lambda)$: thus the path-component of $(H[-n], q_\lambda)$ in $\mathrm{Pn}(\mathcal{D}^p(\bZ), \Qoppa^{[-2n]})$ is equivalent to $B\mathrm{Aut}(H, \lambda)$. This yields a map
\begin{equation*}
inc : B\mathrm{Aut}(H, \lambda) \lra \mathrm{Pn}(\mathcal{D}^p(\bZ), \Qoppa^{[-2n]}) \lra  \mathcal{L}(\bZ, \Qoppa^{[-2n]}) \simeq \Omega^{\infty+2n} \LL^s(\bZ).
\end{equation*}

Suppose now that $(T, \ell)$ is a finite $\bZ$-module $T$ with a $(-1)^n$-symmetric linking form $\ell : T \otimes T \to \bQ/\bZ$ which is nondegenerate, i.e.\ the adjoint $\ell^\mathrm{ad} : T \to \mathrm{Hom}_\bZ(T, \bQ/\bZ)$ is an isomorphism, then there is a Poincar{\'e} object in $(\mathcal{D}^p(\bZ), \Qoppa^{[-(2n-1)]})$ given by $T[-n]$ equipped with the symmetric form
$$q_\ell: T[-n] \otimes^\bL T[-n] \lra T[-n] \otimes T[-n] \overset{\ell}\lra \bQ/\bZ[-2n] \overset{\beta}\lra \bZ[-(2n-1)]$$
where the first map is the truncation and the last is the universal Bockstein. (Before now we have been working in $\mathcal{D}^p(\bZ)$ and all tensor products have been implicitly derived: here we indicate that the first is derived and the second is not.) Again, as $T[-n]$ is concentrated in a single degree the space of automorphisms of $(T[-n], q_\ell)$ is equivalent to the discrete group $\mathrm{Aut}(T, \ell)$, which as above yields a map
\begin{equation*}
inc : B\Aut(T, \ell) \lra \mathrm{Pn}(\mathcal{D}^p(\bZ), \Qoppa^{[-(2n-1)]}) \lra \mathcal{L}(\bZ, \Qoppa^{[-(2n-1)]}) \simeq \Omega^{\infty+2n-1} \LL^s(\bZ).
\end{equation*}

\subsection{Surgery above the middle dimension}\label{sec:SurgeryBelowMid}
Suppose now that $\pi : {E} \to |K|$ is an oriented Poincar{\'e} mock bundle of dimension $d$ which has the property that the inclusions $E_\tau \subset E_\sigma$ are $\bZ$-homology equivalences whenever $\tau \subset \sigma$. (For example, $\pi$ could be a block bundle.) Then the associated functor
$$C : \mathsf{Simp}(K)^{op} \lra \mathcal{D}^p(\bZ)$$
has the property that it sends each morphism to an equivalence. Write $\mathsf{C}_{loc} \subset \mathsf{C}$ for the full subcategory of those functors sending each morphism to an equivalence. It inherits a hermitian structure $\Qoppa_{K, loc}$ from $\Qoppa_{K}$ by restriction, and the associated bilinear functor is still perfect, with duality $D_{K, loc}$ still given by the formula for $D_K$ in \eqref{eq:KDuality}, as this duality preserves the subcategory $\mathsf{C}_{loc}$. Furthermore, on this subcategory the duality simplifies to
\begin{equation}\label{eq:DualityKloc}
D_{K, loc}(X)(\sigma) \simeq \Hom_\bZ(X(\sigma), \bZ),
\end{equation}
because the homotopy limit in \eqref{eq:KDuality} is over a (homotopically) constant diagram.

We will make use of this to do algebraic surgery on $(C, q) \in (\mathsf{C}_{loc}, \Qoppa_{K, loc}^{[-d]})$, using the conveniently-packaged result of \cite[Section 1.3]{No9III}. 

If $d=2n$ then we consider the usual $t$-structure on $\mathcal{D}^p(\bZ)$. The duality $D(-) = \mathrm{Hom}_\bZ(-, \bZ)$ associated to the symmetric Poincar{\'e} structure satisfies $D(\mathcal{D}^p(\bZ)_{\leq 0}) \subset \mathcal{D}^p(\bZ)_{\geq -1}$, as $\bZ$ has global dimension 1, as well as $D(\mathcal{D}^p(\bZ)_{\geq 0}) \subset \mathcal{D}^p(\bZ)_{\leq 0}$. 
Applied objectwise, this yields a $t$-structure on $\mathsf{C}_{loc}$, and the duality $D_{K, loc}$ satisfies $D_{K, loc}((\mathsf{C}_{loc})_{\leq 0}) \subset (\mathsf{C}_{loc})_{\geq -1}$ and  $D_{K, loc}((\mathsf{C}_{loc})_{\geq 0}) \subset (\mathsf{C}_{loc})_{\leq 0}$, as the expression \eqref{eq:DualityKloc} shows that this duality is performed objectwise.\footnote{It is \emph{not} the case that $D_K (\mathsf{C}_{\leq 0}) \subset \mathsf{C}_{\geq -1}$; this is the crucial place where the assumption that each inclusion $E_\tau \subset E_\sigma$ is a $\bZ$-homology equivalence enters.} We apply the algebraic surgery move of \cite[Proposition 1.3.1]{No9III} to the Poincar{\'e} object $(C,q) \in (\mathsf{C}_{loc}, \Qoppa_{K, loc}^{[-2n]})$, using this $t$-structure, $r=\infty$, and $a=0$, to obtain a cobordism to a $(C', q')$ where $C'$ is $(-n)$-connective. As $S^{-2n} \otimes D_{K, loc}(C') \simeq  C'$, it follows that $C'$ is also $(-n)$-truncated, so it is an $n$-fold desuspension of an object in the heart of this $t$-structure, i.e.\ a local coefficient system on $K$. Contemplating the surgery move of \cite[Proposition 1.3.1]{No9III} shows that $H_{-n}(C')$ is the local coefficient system $\sigma \mapsto H^n(E_\sigma ; \bZ)/tors$, and the symmetric form $q'$ on it is that given by cup product and evaluation against the fundamental class of a fibre.

If $d=2n+1$ then we will proceed similarly, but use a modified $t$-structure. Define a $t$-structure on $\mathcal{D}^p(\bZ)$ by declaring that $X \in \mathcal{D}^p(\bZ)_{\geq 0}$ if $H_i(X)=0$ for $i<-1$ and $H_{-1}(X)$ is torsion, and that $X \in \mathcal{D}^p(\bZ)_{\leq 0}$ if $H_i(X)=0$ for $i>0$ and $H_0(X)$ is torsionfree. This is easily checked to define a $t$-structure, and the duality $D(-) = \mathrm{Hom}_\bZ(-, \bZ)$ associated to the symmetric Poincar{\'e} structure satisfies $D(\mathcal{D}^p(\bZ)_{\leq 0}) \subset \mathcal{D}^p(\bZ)_{\geq 0}$. Applied objectwise, this yields a $t$-structure on $\mathsf{C}_{loc}$, and as above the duality $D_{K, loc}$ satisfies $D_{K, loc}((\mathsf{C}_{loc})_{\leq 0}) \subset (\mathsf{C}_{loc})_{\geq 0}$. We apply the algebraic surgery move of \cite[Proposition 1.3.1]{No9III} to the Poincar{\'e} object $(C,q) \in (\mathsf{C}_{loc}, \Qoppa_{K, loc}^{[-(2n+1)]})$, using this $t$-structure, $r=\infty$, and $a=-1$,  to obtain a cobordism to a $(C', q')$ where $C'$ is $(-n)$-connective with respect to this $t$-structure. That is, each chain complex $C'(\sigma)$ has $H_i(C'(\sigma))=0$ for $i < -n-1$ and $H_{-n-1}(C'(\sigma))$ torsion. By the Universal Coefficient Theorem we then have $H_i(D(C'(\sigma))) = 0$ for $i > n$, and as $q'$ induces equivalences $S^{-(2n+1)} \otimes D(C'(\sigma)) \simeq  C'(\sigma)$, it follows that each $C'(\sigma)$ only has homology in degree $-(n+1)$. Thus the $(n+1)$-fold suspension of $C'$ is in the heart of the \emph{ordinary} $t$-structure, i.e.\ is a local coefficient system on $K$. Contemplating the surgery move of \cite[Proposition 1.3.1]{No9III} shows that $H_{-(n+1)}(C')$ is the local coefficient system $\sigma \mapsto tors H^{n+1}(E_\sigma ; \bZ)$, and the symmetric form $q'$ on it is that given by the linking form.

In both cases, the coassembly of the cobordism from $(C, q)$ to $(C', q')$ produces a homotopy $\mathrm{Coass}(C, q) \simeq \mathrm{Coass}(C', q') : |K| \to \mathcal{L}(\bZ, \Qoppa^{[-d]})$, and if $F$ is a typical fibre of $\pi : E \to |K|$ then the discussion above identifies $\mathrm{Coass}(C', q')$ with the composition
$$|K| \overset{\phi}\lra \begin{cases}
B\mathrm{Aut}(H^n(F;\bZ)/tors, \lambda) & d=2n \\
B\mathrm{Aut}(tors\, H^{n+1}(F;\bZ), \ell) & d=2n+1
\end{cases} \xrightarrow{inc} \Omega^{\infty+d} \LL^s(\bZ).$$
The identifies the family signature $\mathrm{Coass}(C, q)$ with the lower composition in the statement of Theorem \ref{thm:FSTLThy}.

\subsection{Proof of Theorem \ref{thm:FSTLThy}}
Suppose now that $\pi : E \to |K|$ is a block bundle, or more generally a mock bundle such that the inclusions $E_\tau \subset E_\sigma$ are $\bZ$-homology equivalences whenever $\tau \subset \sigma$, classified by $f : |K| \to |\Mock(d)|$. There is a diagram
\begin{equation*}
\begin{tikzcd}
{|K|} \arrow[rd, "f"] \arrow[rddd, bend right=20, swap, "{\mathrm{Coass}(C, q)}"] \arrow[rrrd, bend left=20, "\alpha"]\\
& {|\Mock(d)|} \dar \arrow[rr, "\simeq"] && {\Omega^{\infty+d} \MSTop} \ar[dd, "{\Omega^{\infty+d} \sigma}"] \\
& {|\Mock^P(d)|} \arrow[d, "{\mathrm{Coass}(C_d, q_d)}"] \arrow[rrd, bend left=15, "{\Omega^{\infty+d} \sigma^P}"]  \\
& {\mathcal{L}(\bZ, \Qoppa^{[-d]})} \arrow[rr, "\simeq"] & & \Omega^{\infty+d} \LL^s(\bZ)
\end{tikzcd}
\end{equation*}
in which each region commutes (up to homotopy) tautologically, and by the discussion in the last section the anticlockwise composition agrees with $inc \circ \phi$, proving Theorem \ref{thm:FSTLThy}.

\begin{remark}\label{rem:RanickiBookIssues}
The main difficulty in implementing the above in the framework of \cite{RanickiBook} is that Ranicki only considers bounded chain complexes of finitely-generated free (or projective) modules, and (co)chains on a manifold is never equal to such a thing. The usual solution to this issue seems to have been to pretend otherwise, or to change definitions but assume that Ranicki's results are unchanged, or to restrict to PL manifolds and use simplicial chains for some choice of PL-triangulation instead. 
\end{remark}

\section{Twisted signature formulas in $L$-theory}\label{sec:Addenda}

We consider Theorem \ref{thm:FSTLThy} to be the most natural formulation of the Family Signature Theorem, but for applications we must also provide tools for understanding the bottom map $inc$, just as we did in Section \ref{sec:TwistedFormulaKO} for the $KO[\tfrac{1}{2}]$-theory formulation. This will require significantly more external machinery than the discussion so far. In particular it relies extensively on the relation between $L$-theory and Grothendieck--Witt theory developed in \cite{No9I, No9II, No9III}, and we must now assume more familiarity with these papers.

\subsection{The symmetric $L$-theory spectrum}\label{sec:LThyBG}

Following Hebestreit, Land, and Nikolaus \cite{HLN}, the symmetric $L$-theory spectrum of the integers may be described as follows: there is a fibration sequence
\begin{equation}\label{eq:dRsequence}
\mathrm{dR} \lra \LL^s(\bZ) \lra \LL^s(\bR),
\end{equation}
where $\pi_*(\LL^s(\bR)) = \bZ[x^{\pm 1}]$ with $|x|=4$ and $\pi_i(\mathrm{dR})$ is $\bZ/2$ if $i \equiv 1 \mod 4$ and is zero otherwise. The homotopy groups in degrees $4i$ detect the signature, and those in degrees $4i+1$ detect the de Rham invariant.

There are equivalences $\LL^s(\bZ)[\tfrac{1}{2}] \simeq \LL^s(\bR)[\tfrac{1}{2}] \simeq \KO[\tfrac{1}{2}]$, and we would like to say that Theorem \ref{thm:FSTSullivan} is obtained from Theorem \ref{thm:FSTLThy} by inverting 2 and using such an equivalence. In order for the Sullivan orientation $\Delta_{Top}$ to agree (up to phantom maps) with the Ranicki orientation $\sigma$ under such an equivalence, it suffices (by construction of $\Delta_{Top}$) to choose an equivalence inducing the ring map $a \mapsto x : \bZ[\tfrac{1}{2}][a^{\pm 1}] = \pi_*(\KO[\tfrac{1}{2}]) \to \pi_*(\LL^s(\bR)[\tfrac{1}{2}]) = \bZ[\tfrac{1}{2}][x^{\pm 1}]$ on homotopy groups. Choosing this isomorphism between their homotopy groups we obtain an isomorphism of generalised homology theories
$$KO[\tfrac{1}{2}]_*(-) \overset{\sim}\leftarrow MSO_*(-) \underset{MSO_*}\otimes \bZ[\tfrac{1}{2}][a^{\pm 1}] \cong MSO_*(-) \underset{MSO_*}\otimes \bZ[\tfrac{1}{2}][x^{\pm 1}] \overset{\sim}\to L^s(\bR)[\tfrac{1}{2}]_*(-),$$
as the right hand map can be shown to be an isomorphism using the Landweber exact functor theorem \cite[Example 3.4]{Landweber} just as the left hand map is. This gives an equivalence of representing spectra $\KO[\tfrac{1}{2}] \simeq \LL^s(\bR)[\tfrac{1}{2}]$ inducing the required map on homotopy groups. 

At the prime 2, following Taylor and Williams \cite[Section 2]{TaylorWilliams} one can use that $\MSO_{(2)}$ is a generalised Eilenberg--MacLane spectrum to deduce that any module spectrum over it is too. This applies to $\LL^s(\bZ)_{(2)}$ and $\LL^s(\bR)_{(2)}$ via the map of homotopy ring\footnote{In our construction we have not justified that $\sigma$ is a map of homotopy ring spectra. We can appeal to e.g.\ \cite{LauresMcClure} for this.} spectra $\MSO \to \MSTop \overset{\sigma}\to \LL^s(\bZ)$, and hence to $\mathrm{dR} = \mathrm{dR}_{(2)}$ using \eqref{eq:dRsequence}. Taylor and Williams produce specific maps
\begin{align*}
L: \LL^s(\bZ)_{(2)} &\lra \bigoplus_{i \in \bZ} \H\bZ_{(2)}[4i]\\
r: \LL^s(\bZ)_{(2)} &\lra \bigoplus_{j \in \bZ} \H\bZ/2[4j+1]
\end{align*}
which combine to give an equivalence. Pulled back along the Ranicki orientation they behave as follows. The class $\sigma^*L$ corresponds under the Thom isomorphism with Morgan and Sullivan's class $\mathscr{L} \in H^*(BSTop ; \bZ_{(2)})$ constructed in \cite[Section 7]{MorganSullivan}, which is in turn characterised by restricting to the inverse of the Hirzebruch $L$-class when rationalised, and the square of the total Wu class\footnote{Characterised by $\Sq(V) = 1/w$ for $w \in H^*(BSTop;\bZ/2)$ the total Stiefel--Whitney class. It has the form $V = 1 + V_2 + V_4 + \cdots$; the odd components vanish, by the argument of p.\ 483 of \cite{MorganSullivan}, using topological manifolds whose Gauss maps $M \to BSTop$ are highly connected.} $V$ when reduced modulo 2. The class $\sigma^*r$ corresponds under the Thom isomorphism with $\sum_{i \geq 0}V_{2i} \cdot \Sq^1V_{2i}$, which is the well-known characteristic class measuring the de Rham invariant \cite{LMP}. The composition $\mathrm{dR} \to \LL^s(\bZ) \to \LL^s(\bZ)_{(2)} \overset{r}\to \bigoplus_{j \in \bZ} \H \bZ/2[4j+1]$ is necessarily an equivalence, which gives a splitting $\LL^s(\bZ) \simeq \LL^s(\bR) \oplus \mathrm{dR}$ such that $\LL^s(\bR) \to \LL^s(\bZ) \to \LL^s(\bZ)_{(2)} \overset{L}\to \bigoplus_{i \in \bZ} \H \bZ_{(2)}[4i]$ is a 2-local equivalence.

A more structured approach is also possible. Hebestreit, Land, and Nikolaus describe \cite[Section 3]{HLN} an $E_1$-algebra map $\H \bZ_{(2)} \to \MSO_{(2)}$, which then endows any $\MSO_{(2)}$-module with a canonical $\H \bZ_{(2)}$-module structure. Using \cite[Corollary 3.8]{HLN} it follows that there are unique equivalences of $\H \bZ_{(2)}$-modules
\begin{equation}\label{eq:HZ2ModEqs}
\LL^s(\bR)_{(2)} \simeq \bigoplus_{i \in \bZ} \H\bZ_{(2)}[4i] \quad\quad\quad \mathrm{dR} \simeq \bigoplus_{j \in \bZ} \H\bZ/2[4j+1]
\end{equation}
inducing the identity on homotopy groups. They also show \cite[Corollary 4.3]{HLN} that there is a unique splitting $\LL^s(\bR) \to \LL^s(\bZ)$ of $E_1$-algebras which 2-locally is a splitting of $\H \bZ_{(2)}$-algebras.

\begin{remark}
 We do not assert anything about any compatibility between the splittings of the previous paragraph and those obtained using the maps $L$ and $r$ above. Using the discussion in \cite[Section 2]{TaylorWilliams} it seems that this kind of question comes down to whether the class $\mathscr{L} \in H^*(BSO;\bZ_{(2)})$ is trivial on the bundle classified by $\eta: \tau_{\geq 2}\Omega^2 S^3 \to BSO$.
\end{remark}

\subsection{Odd dimensions}\label{sec:AddendaOdd}

In Section \ref{sec:MapInc} we have explained how a nondegenerate $(-1)^n$-symmetric linking form $(T,\ell)$ yields a Poincar{\'e} object $(T[-n], q_\ell)$ and hence a map
\begin{equation}\label{eq:incMap}
inc : B\Aut(T, \ell) \lra \mathcal{L}(\bZ, \Qoppa^{[-(2n-1)]}) = \Omega^{\infty+2n-1} \LL^s(\bZ).
\end{equation}

\begin{theorem}\label{thm:LinkingTriv}
The map \eqref{eq:incMap} is homotopic to a constant map.\footnote{Though is not in general nullhomotopic, as when $2n+1 \equiv 1 \mod 4$ it will land in the path-component dictated by the de Rham invariant of $T$, which can be nontrivial.}
\end{theorem}
\begin{proof}
Firstly, there is a canonical decomposition $T = \bigoplus_{p \text{ prime}} T_p$ of $T$ into its Sylow subgroups, which is is orthogonal with respect to $\ell$. This gives a corresponding decomposition of $\mathrm{Aut}(T,\ell)$, meaning that we may assume that $T$ is a $p$-group.

Secondly, if $(T,\ell)$ is a nondegenerate $(-1)^{n}$-symmetric linking form on a $p$-group then consider the subgroup $L := (p \cdot T) \cap ([p](T))$ of $p$-divisible and $p$-torsion elements of $T$. The restriction of $\ell$ to $L$ vanishes, as if $x = p \cdot \bar{x}, y \in L$ then $\ell(x, y) = \ell(p \bar{x}, y) = \ell(\bar{x}, py)=0$ as $y$ is $p$-torsion. Thus we may do algebraic surgery to $(T[-n], q_\ell)$ along the map $L[-n] \to T[-n]$, the result of which is the Poincar{\'e} object associated to the linking form $(T', \ell')$ with
$$T' := \mathrm{ker}(\ell^\mathrm{ad} : T/L \to \mathrm{Hom}_\bZ(L, \bQ/\bZ))$$
and $\ell'$ the linking form induced by $\ell$. This is functorial, giving a map $B\mathrm{Aut}(T, \ell) \to B\mathrm{Aut}(T', \ell')$, and furthermore the algebraic surgery between these forms is functorial, giving a homotopy between the two maps into $\Omega^{\infty+2n-1} \LL^s(\bZ)$. If $L \neq 0$ then $T'$ is strictly smaller than $T$, so continuing in this way we may assume that $(T,\ell)$ is such that $L=0$, i.e.\ $T$ is an elementary abelian $p$-group.

When $(T,\ell)$ is such that $T$ is an elementary abelian $p$-group, $\ell$ takes values in the subgroup of $\bQ/\bZ$ generated by $\tfrac{1}{p}$, and identifying this subgroup with $\bF_p$ we may consider $\ell : T \otimes T \to \bF_p$ as a nondegenerate $(-1)^n$-symmetric form on an $\bF_p$-module. Thus $\mathrm{Aut}(T, \ell)$ is a symplectic or orthogonal group over $\bF_p$, and there is a factorisation
$$inc : B\mathrm{Aut}(T, \ell) \lra \Omega^{\infty+2n} \LL^s(\bF_p) \overset{\partial_p}\lra \Omega^{\infty+2n-1} \LL^s(\bZ)$$
where the latter is the map arising in the localisation-d{\'e}vissage sequence
\begin{equation}\label{eq:locsequence}
\bigoplus_{p \text{ prime}} \Sigma^{-1} \LL^s(\bF_p) \overset{\oplus \partial_p}\lra \LL^s(\bZ) \lra \LL^s(\bQ).
\end{equation}

Thirdly, suppose that $p$ is odd. We will argue that $\partial_p : \Sigma^{-1} \LL^s(\bF_p) \to \LL^s(\bZ)$ is nullhomotopic. As $\partial_p$ has a canonical nullhomotopy when mapped to $\LL^s(\bQ)$ it also does when mapped to $\LL^s(\bR)$, giving a canonical lift $\partial_p' : \Sigma^{-1} \LL^s(\bF_p) \to \mathrm{dR}$ which we wish to show is nullhomotopic. 

The target of this map is 2-local, so we can localise all spectra involved at 2 and make use of the canonical $H\bZ_{(2)}$-module structure on $\LL^s(\bZ)_{(2)}$-modules discussed in Section \ref{sec:LThyBG}. We have already discussed how the terms in \eqref{eq:dRsequence} obtain $H\bZ_{(2)}$-module structures, in particular giving an equivalence $\mathrm{dR} \simeq \bigoplus_{i \in \bZ} \H \bZ/2 [4i+1]$ of $H\bZ_{(2)}$-modules. On the other hand the fibre sequence \eqref{eq:locsequence} endows $\Sigma^{-1} \LL^s(\bF_p)$ with a $\LL^s(\bZ)$-module structure (potentially different from that given by $\LL^s(\bZ) \to \LL^s(\bF_p)$, though experts tell me it is in fact not) making $\partial_p$ into an $\LL^s(\bZ)$-module map. As $p$ is odd we have
$$\pi_i(\LL^s(\bF_p)) = \begin{cases}
W(\bF_p) & i \equiv 0 \mod 4\\
0 & \text{else}
\end{cases}$$
where $W(F_p)$ is the Witt group of $\bF_p$ and is either $\bZ/4$ or $\bZ/2 \oplus \bZ/2$ \cite[Proposition 4.3.2]{RanickiPhoneBook}. In either case it follows that $\Sigma^{-1} \LL^s(\bF_p)$ is 2-local and with the $\H\bZ_{(2)}$-module structure given by \eqref{eq:locsequence} it is equivalent to $\bigoplus_{j \in \bZ} \H W(\bF_p)[4j-1]$. As the nullhomotopy of $\partial_p$ into $\LL^s(\bR)$ can be taken to be one of $\LL^s(\bZ)$-modules, via \eqref{eq:locsequence}, the map $\partial_p' : \Sigma^{-1} \LL^s(\bF_p) \to \mathrm{dR}$ is one of $\H\bZ_{(2)}$-modules with the module structures just described. But
$$0=\left[\bigoplus_{j \in \bZ} \H W(\bF_p)[4j-1], \, \bigoplus_{i \in \bZ} \H \bZ/2 [4i+1]\right]_{\H \bZ_{(2)}\text{-}mod}.$$
 so $\partial'_p$ is null.

Fourthly, suppose that $p=2$. In this case the map $\partial_2'$ is not trivial, and we instead argue that the map $B\mathrm{Aut}(T, \ell) \to \Omega^{\infty+2n} \LL^s(\bF_2)$ is homotopic to a constant map, so that its composition with $\Omega^{\infty+2n} \LL^s(\bF_2) \to \Omega^{\infty+2n-1} \mathrm{dR}$ is too. As we start with actual automorphisms of a symmetric form over $\bF_2$, this map factors through the Grothendieck--Witt space $\Omega^\infty \mathrm{GW}(\bF_2 ; \Qoppa)$, via the map
\begin{equation}\label{eq:GWmap}
\Omega^\infty \mathrm{GW}(\bF_2 ; \Qoppa) \lra \Omega^\infty \LL(\bF_2 ; \Qoppa) = \Omega^\infty \LL^s(\bF_2) = \Omega^{\infty+2n} \LL^s(\bF_2),
\end{equation}
where the latter identification holds because $\LL^s(\bF_2)$ is 2-periodic. (Here $\Qoppa = \Qoppa^s$ is the symmetric structure.) By \cite[Proposition 3.1.4]{No9III} the map
$$\mathrm{GW}(\bF_2 ; \Qoppa) \lra \K(\bF_2 ; \Qoppa)^{hC_2}$$
is an equivalence but by Quillen's calculation \cite{QuillenFF} of the $K$-theory of finite fields, the truncation map $ \K(\bF_2 ; \Qoppa) \to \H \bZ$ is a 2-local equivalence, so $\Omega^\infty \mathrm{GW}(\bF_2 ; \Qoppa)$ is 2-locally equivalent to the discrete space $\bZ$. As the target of \eqref{eq:GWmap} is 2-local, it follows that \eqref{eq:GWmap} is homotopic to a constant map as required.
\end{proof}

\begin{corollary}\label{cor:OddFSIsTrivial}
Let $\pi: E \to |K|$ be an oriented topological block bundle with $d$-dimensional fibre, and $d$ odd. Then the composition
$$|K| \overset{\alpha}\lra \Omega^{\infty+d} \MSTop \overset{\Omega^{\infty+d}\sigma}\lra \Omega^{\infty+d} \LL^s(\bZ)$$
is homotopic to a constant map.
\end{corollary}

\subsection{Even dimensions}\label{sec:AddendaEven}

With the choice of equivalence $\LL^s(\bR)[\tfrac{1}{2}] \simeq \KO[\tfrac{1}{2}]$ explained in Section \ref{sec:LThyBG}, the composition
$$B\Aut(H^n(F;\bR), \lambda) \overset{inc}\lra \Omega^{\infty+2n} \LL^s(\bR) \lra \Omega^{\infty+2n} \LL^s(\bR)[\tfrac{1}{2}] \simeq \Omega^{\infty+2n} \KO[\tfrac{1}{2}]$$
agrees (up to phantom maps) with the map $\sign$ constructed in Section \ref{sec:TwistedSignaturesKO}, because it has the same interpretation in terms of signatures on $\bZ/k$-bordism.

The new information in Theorem \ref{thm:FSTLThy} is therefore at the prime $2$. By the discussion in Section \ref{sec:LThyBG}, there is a splitting $\LL^s(\bZ) \simeq \LL^s(\bR) \oplus \mathrm{dR}$ induced by the composition $\mathrm{dR} \to \LL^s(\bZ) \overset{r}\to \bigoplus_{j \in \bZ} \H \bZ/2[4j+1]$ being an equivalence, and this induces an equivalence $\LL^s(\bR)_{(2)} \simeq \bigoplus_{i \in \bZ} \H \bZ_{(2)}[4i]$ such that the composition
$$\MSTop \overset{\sigma}\lra \LL^s(\bZ) \lra \LL^s(\bR)_{(2)} \simeq \bigoplus_{i \in \bZ} \H \bZ_{(2)}[4i]$$
corresponds under the Thom isomorphism to Morgan and Sullivan's class $\mathscr{L}$. The following theorem is the analogue of Theorem \ref{thm:KOSigFormula}, and stating it requires a preparatory lemma.

\begin{lemma}\label{lem:DefPHTilde}
There are unique classes $\widetilde{\mathrm{ph}}_i \in H^{4i}(BO;\bZ_{(2)})$ rationalising to $2^{2i} \mathrm{ph}_i$ and for $i>0$ reducing to zero modulo 2. 

Furthermore, for $i>0$ these classes are in fact divisible by 4.
\end{lemma}
\begin{proof}
As the integral cohomology of $BO$ only has $\bZ/2$-torsion the uniqueness is clear. For existence, as in Corollary \ref{cor:SigMultMeyer} we use that $2^j/j!$ is 2-integral and even 2-integrally divisible by 2 so that $\widetilde{\mathrm{ch}}_j = 2^j \mathrm{ch}_j$ is a 2-integral cohomology class on $BU$ which reduces to zero modulo 2. Pulling $\widetilde{\mathrm{ch}}_{2i}$ back to $BO$ gives the required classes $\mathrm{ph}_i$. The second part follows identically to the case $n$ even of the proof of Corollary \ref{cor:SigMultMeyer}.
\end{proof}

To state the following, we use the real form $\xi_\bR$ of $\xi$ from Remark \ref{rem:RealXi}..

\begin{theorem}\label{thm:2LocalPontrjaginFormula}
If $(H_\bR, \lambda)$ is a nondegenerate $(-1)^n$-symmetric bilinear form then the square
\begin{equation*}
\begin{tikzcd}
B\Aut(H_\bR, \lambda) \arrow[rr, "inc"] \dar{\xi_\bR} && \Omega^{\infty+1-(-1)^n} \LL^s(\bR)_{(2)} \dar{\simeq}   \\
\Omega^{\infty + 1 - (-1)^n} \KO \arrow[rr, "\Omega^{1-(-1)^n}\widetilde{\mathrm{ph}}"] & & \Omega^{\infty+1 - (-1)^n} \left(\bigoplus_{i \in \bZ} \H \bZ_{(2)}[4i] \right)
\end{tikzcd}
\end{equation*}
commutes up to homotopy and phantom maps.
\end{theorem}
\begin{proof}
We can verify this after inverting 2, and after completing at 2. After inverting 2 we are working rationally, and the equivalence
\begin{equation}\label{eq:RatEq}
\KO_{(0)} \simeq \LL^s(\bR)_{(0)} \simeq \bigoplus_{i \in \bZ} \H \bQ[4i]
\end{equation}
obtained by further localising the equivalences $\LL^s(\bR)[\tfrac{1}{2}] \simeq \KO[\tfrac{1}{2}]$ and $\LL^s(\bR)_{(2)} \simeq \bigoplus_{i \in \bZ} \H \bZ_{(2)}[4i]$ sends $a^k \in \pi_{4k}(\KO)_{(0)}$ to $1 \in \pi_{4k}(\bigoplus_{i \in \bZ} \H \bQ[4i])$. As $a$ was chosen to map to the square of the Bott class $b^2 \in \pi_4(\K)$ under complexification, and the Chern character takes the value 1 on $b$, it follows that the equivalence \eqref{eq:RatEq} is induced by the Pontrjagin character. By Theorem \ref{thm:KOSigFormula} the clockwise composition is then rationally $\mathrm{ph}(\tfrac{1}{2}r(b^n \psi^2 \xi))$ which is calculated in the proof of that theorem to be $\mathrm{ch}(\psi^2 \xi)$ which is the same as $\widetilde{\mathrm{ph}}(\xi_\bR)$.

To deal with the 2-complete case we use the fact that, up to a translation of path-components, the maps $inc$ and $\xi_\bR$ commute with the map $B\Aut(H_\bR, \lambda) \to B\Aut(H_\bR \oplus H'_\bR, \lambda \oplus \lambda')$ given by stabilising by a form $(H'_\bR, \lambda')$. As the square does commute at the level of $\pi_0$ by the previous case, by choosing $(H'_\bR, \lambda')$ to have large rank (and opposite signature to $(H_\bR, \lambda)$ if $n$ is even) this allows us to assume that $(H_\bR, \lambda)$ is a hyperbolic form of arbitrarily large rank. Then, by a version \cite[p.\ 260]{KaroubiRig} of the stable Milnor conjecture for the groups $\mathrm{Sp}_{2g}(\bR)$ or\footnote{Karoubi writes $\mathrm{O}_{2n}(\bR)$ but this seems to mean what we call $\mathrm{O}_{n,n}(\bR)$, the automorphism group of the hyperbolic form.} $\mathrm{O}_{g,g}(\bR)$ the map
$$B\Aut(H_\bR, \lambda) \lra B\Aut(H_\bR, \lambda)^{top} \simeq \begin{cases}
BU(g) & n \text{ odd}\\
BO(g) \times BO(g) & n \text{ even}
\end{cases}$$
induces an isomorphism on cohomology with all finite coefficients in a range of cohomological degrees tending to $\infty$ with $g$. It therefore also induces an isomorphism with 2-adic coefficients in such a range. 

If $n$ is odd it follows that in each degree in the stable range $H^*(B\Aut(H_\bR, \lambda) ; \bZ_2)$ is a finitely-generated free $\bZ_2$-module, so by the Bockstein sequence the map
$$H^*(B\Aut(H_\bR, \lambda) ; \bZ_2) \lra H^*(B\Aut(H_\bR, \lambda) ; \bQ_2)$$
is injective. As the diagram in the statement of the theorem commutes over $\bQ$ it also does over $\bQ_2$, and by this injectivity it also does over $\bZ_2$. 

If $n$ is even it follows that in the stable range $H^*(B\Aut(H_\bR, \lambda) ; \bZ_2)$ is a sum of a finitely-generated free $\bZ_2$-module and a finite $\bZ/2$-module. Elements are therefore detected by their images with $\bQ_2$- and $\bZ/2$-coefficients. With $\bQ_2$-coefficients we proceed as above. With $\bZ/2$-coefficients the composition $\Omega^{1-(-1)^n} \widetilde{\mathrm{ph}} \circ \xi_\bR$ is trivial, as the classes $\widetilde{\mathrm{ph}}_i$ are trivial modulo 2. To show the other composition is trivial, we use that
$$\bZ/2[w_1, w_2, \ldots w_g, w'_1, \ldots, w_g] \overset{\sim}\to H^*(B\Aut(H_\bR, \lambda)^{top};\bZ/2) \to H^*(B\Aut(H_\bR, \lambda);\bZ/2)$$
is an isomorphism in the stable range, and that this is detected on the subgroup
$$\{\pm 1\}^g \times \{\pm 1\}^g \subset \Aut(H_\bR, \lambda)$$
given by acting by a sign in each basis vector, in some basis given by bases for a choice of positive and negative definite subspaces. As the map $inc$ is additive with respect to orthogonal sum of forms, this reduces us to considering
$$inc: BC_2 = B\Aut(\bR, (1)) \lra \Omega^{\infty+2n} \LL^s(\bR),$$
or its negative-definite analogue, which can be treated similarly. To evaluate this we can use the family signature theorem, because the complex conjugation involution on $\mathbb{CP}^2$ gives a splitting $BC_2 \to B\mathrm{Diff}^+(\mathbb{CP}^2) \overset{\phi} \to B\Aut(\bR, (1))$ and so we may apply Theorem \ref{thm:FSTLThy} to say that the map in question is homotopic to
$$BC_2 \to \Omega^\infty \mathrm{Th}(-T_\pi E \to E) \to \Omega^{\infty+4} \MSTop \overset{\sigma}\lra \Omega^{\infty+4} \LL^s(\bZ) \overset{L}\lra \Omega^{\infty+4} \left(\bigoplus_{i \in \bZ} \H \bZ_{(2)}[4i] \right),$$
where the first map is the parameterised Pontrjagin--Thom map for the smooth oriented fibre bundle $\pi: E:= EC_2 \times_{C_2} \mathbb{CP}^2 \to BC_2$, and the second is given by Thomifying the map $E \to BSTop$ classifying $-T_\pi E$. As the reduction of $\sigma^*(L)$ modulo 2 corresponds under the Thom isomorphism to the square of the total Wu class, as a spectrum cohomology class on $\mathrm{Th}(-T_\pi E \to E)$ it corresponds to $u_{-T_\pi E} \cdot V(-T_\pi E)^2$. This satisfies
\begin{align*}
\Sq(u_{-T_\pi E} \cdot V(-T_\pi E)^2) &= \Sq(u_{-T_\pi E}) \cdot (\Sq(V(-T_\pi E)))^2\\
 &= u_{-T_\pi E} \cdot w(-T_\pi E) \cdot (\tfrac{1}{w(-T_\pi E)})^2\\
 &= u_{-T_\pi E} \cdot w(T_\pi E).
\end{align*}
This cohomology class is supported in degrees $\leq 0$, so when pulled back to $\Sigma^\infty_+ BC_2$ along the adjoint of the parameterised Pontrjagin--Thom map it only has a component of degree 0 (namely, the characteristic number $\int_{\mathbb{CP}^2} w_4(T\mathbb{CP}^2) = 1$). But as $\Sq$ is invertible, the pullback of $u_{-T_\pi E} \cdot V(-T_\pi E)^2$ to $\Sigma^\infty_+ BC_2$ must also be trivial in positive degrees, as required.
\end{proof}

The remaining information at the prime 2 concerns the de Rham invariant map $r : \LL^s(\bZ) \to \bigoplus_{j \in \bZ} \H\bZ/2[4j+1]$. We have already mentioned that $r \circ \sigma$ corresponds, under the Thom isomorphism, with the class $\sum_{i \geq 0} V_{2i} \cdot \Sq^1 V_{2i}$. Theorem \ref{thm:FateofdR} below determines the corresponding cohomology classes on the intersection form side. During the proof we will explain that $H^3(BSp_\infty(\bZ);\bZ/2) = \bZ/2$ and the unique nontrivial element $r_3$ of this group provides, by stabilising, a characteristic class for all local systems of skew-symmetric lattices. Relatedly, we will explain why the function $r_1 : \Aut(H_\bZ, \lambda) \to \bZ/2$, which assigns to an automorphism $\phi$ the number $\dim_{\bZ/2} (\bZ/2 \otimes tors\tfrac{H_\bZ}{(\mathrm{Id}-\phi)H_\bZ})$ modulo 2, is a homomorphism.

\begin{theorem}\label{thm:FateofdR}
If $n$ is even so $(H_\bZ, \lambda)$ is symmetric then 
$$B\Aut(H_\bZ, \lambda) \overset{inc}\lra \Omega^{\infty+2n}\LL^s(\bZ) \xrightarrow{\Omega^{\infty+2n} r} \Omega^{\infty+2n} \left(\bigoplus_{j \in \bZ} \H\bZ/2[4j+1] \right)$$
is the cohomology class given by $r_1$ in degree 1, and is trivial in all other degrees.

If $n$ is odd so $(H_\bZ, \lambda)$ is skew-symmetric then 
$$B\Aut(H_\bZ, \lambda) \overset{inc}\lra \Omega^{\infty+2n}\LL^s(\bZ) \xrightarrow{\Omega^{\infty+2n} r} \Omega^{\infty+2n} \left(\bigoplus_{j \in \bZ} \H\bZ/2[4j+1] \right)$$
is the cohomology class $r_3$ in degree 3, and is trivial in all other degrees.
\end{theorem}

To begin the proof of this theorem, the map $inc$ tautologically factors over ($\Omega^\infty$ of) the map 
$$\mathrm{bord} : \GW^{s}(\bZ; (-1)^n) \lra \LL^{s}(\bZ; (-1)^n)$$
so it suffices to analyse the map
$$\Omega^\infty_0 \GW^{s}(\bZ; (-1)^n) \xrightarrow{\Omega^\infty \mathrm{bord}} \Omega^{\infty+2n}_0 \LL^s(\bZ) \xrightarrow{\Omega^{\infty+2n} r} \Omega^{\infty+2n} \left(\bigoplus_{j \in \bZ} \H\bZ/2[4j+1] \right),$$
which is an infinite loop map. The following lemma shows that there are not so many such maps which refine to infinite loop maps. 

\begin{lemma}\label{lem:CohIsSmall}
The cohomology suspension map
$$H^*(\tau_{>0} \GW^{s}(\bZ; (-1)^n);\bZ/2) \lra H^*(\Omega^\infty_0 \GW^{s}(\bZ; (-1)^n);\bZ/2)$$
has trivial image in degrees congruent to $1 - 2n \mod 4$, except for degree 1 if $n$ is even or degree 3 if $n$ is odd.
\end{lemma}
\begin{proof}
Berrick and Karoubi \cite{BK} establish a 2-adically cartesian square
$$
\begin{tikzcd}
\tau_{> 0} \GW^{s}_{cl}(\bZ[\tfrac{1}{2}]; (-1)^n) \rar \dar& \tau_{> 0}\GW^{top}(\bR; (-1)^n) \dar\\
\tau_{> 0} \GW^{s}_{cl}(\bF_3; (-1)^n) \rar & \tau_{> 0}\GW^{top}(\bC; (-1)^n)
\end{tikzcd}
$$
where the left-hand terms are (the 0-connected covers of) the classical symmetric Grothendieck--Witt spectra. There is a zig-zag
$$ \GW^{s}_{cl}(\bZ[\tfrac{1}{2}]; (-1)^n) \lra \tau_{\geq 0} \GW^{s}(\bZ[\tfrac{1}{2}] ; (-1)^n) \longleftarrow \tau_{\geq 0} \GW^{s}(\bZ ; (-1)^n)$$
whose left-hand map is an equivalence by \cite[1.3.15]{No9III} (using \cite[Theorem A]{HS}), and whose right-hand map is a 2-adic equivalence by \cite[3.1.11]{No9III}. Combined with Theorems IV.2.4 and IV.5.4 of \cite{FP} this gives fibre sequences of (implicitly 2-completed) spectra
\begin{equation}\label{eq:GWQuillenSeq}
\begin{split}
&\tau_{> 0} \GW^{s}(\bZ; +) \lra bo \oplus bo \lra bso \\
&\tau_{> 0} \GW^{s}(\bZ; -) \lra bu \lra bsp
\end{split}
\end{equation}
where the right-hand maps are $bo \oplus bo \overset{+}\to bo \overset{\psi^3-1}\to bso$ and $bu \to bsp \overset{\psi^3-1}\to bsp$ respectively. These spectra are the deloopings of $BO$, $BSO$, $BU$ and $BSp$, and their $\bZ/2$-cohomology as left modules over the Steenrod algebra $\mathcal{A}$ follows from the calculations of \cite{StongBO}, as
\begin{align*}
H^*(bo) &= \mathcal{A}\{\iota_1\}/ (\Sq^2 \iota_1) &\quad\quad H^*(bso) &=  \mathcal{A}\{\iota_2\}/ (\Sq^3 \iota_2)\\
H^*(bu) &= \mathcal{A}\{\iota_2\}/ (\Sq^1 \iota_2, \Sq^3 \iota_2) &\quad\quad H^*(bsp) &= \mathcal{A}\{\iota_4\} / (\Sq^1 \iota_4, \Sq^5 \iota_4),
\end{align*}
where $\iota_r$ denotes a class of degree $r$, and here and in the rest of this proof all cohomology is taken with $\bZ/2$-coefficients.

The map $\psi^3-1 : bsp \to bsp$ induces multiplication by 8($=3^2-1$) on the lowest homotopy group $\pi_4(bsp) = \bZ_2$, so by the above it induces the zero map on $\bZ/2$-cohomology. 

The based map $BO(1) \overset{L-1}\to BO \overset{\psi^3-1}\to BSO$, where $L$ is the tautological real line bundle over $BO(1)$, is nullhomotopic as $\psi^3(L) = L^{\otimes 3} = L$. Thus the adjoint $\Sigma^\infty BO(1) \to bo \overset{\psi^3-1}\to bso$ is nullhomotopic too, but the first map sends the generator $\iota_1 \in H^1(bo)$ to $w_1(L) \in H^1(BO(1))$ and so sends $H^2(bo) = \bZ/2\{\Sq^1 \iota_1\}$  isomorphically to $H^2(BO(1)) = \bZ/2\{w_1(L)^2\}$. Thus the map $\psi^3-1 : bo \to bso$ is zero on $H^2(-)$, and so on all $\bZ/2$-cohomology.

We obtain extensions of $\mathcal{A}$-modules
\begin{align*}
 \mathcal{A}\{\iota_1\}/  (\Sq^3 \iota_1) \longleftarrow & H^*(\tau_{> 0} \GW^{s}(\bZ; +)) \longleftarrow \mathcal{A}\{\iota_1'\}/ (\Sq^2 \iota_1') \oplus  \mathcal{A}\{\iota_1''\}/ (\Sq^2 \iota_1'') \\
 \mathcal{A}\{\iota_3\}/ (\Sq^1 \iota_3, \Sq^5 \iota_3) \longleftarrow & H^*(\tau_{> 0} \GW^{s}(\bZ; -)) \longleftarrow \mathcal{A}\{\iota_2\}/(\Sq^1 \iota_2, \Sq^3 \iota_2) 
\end{align*}
and so find that $H^*(\tau_{\geq 0} \GW^{s}(\bZ; +))$ is generated as an $\mathcal{A}$-module by three elements of degree 1, and $H^*(\tau_{\geq 0} \GW^{s}(\bZ; -))$ is generated as an $\mathcal{A}$-module by an element of degree 2 and an element of degree 3.\

In the second case the right-hand term vanishes in total degree 3, so there is a unique lift $\bar{\iota}_3$ of the generator $\iota_3$ of the left-hand term. This must satisfy either $\Sq^1(\bar{\iota}_3)=0$ or $\Sq^1(\bar{\iota}_3)= \Sq^2 (\iota_2)$. The latter can be ruled out by assuming that this is the case and then using the Adams spectral sequence to calculate $\pi_3(\GW^{s}(\bZ; -))$, which one finds to be $\bZ/4$. But it is in fact seen to be $\bZ/16$ by calculating with \eqref{eq:GWQuillenSeq} (cf.\ the corresponding table in \cite[Section 3.2]{No9III}). Writing $uns M$ for the unstable quotient of an $\mathcal{A}$-module $M$, we therefore find that
\begin{align*}
uns H^*(\tau_{> 0} \GW^{s}(\bZ; +)) &\cong uns \mathcal{A}\{\iota_1, \iota_1', \iota_1''\}\\
uns H^*(\tau_{> 0} \GW^{s}(\bZ; -)) &\cong uns \mathcal{A}\{\iota_2, \bar{\iota}_3\}/(\Sq^1 \iota_2, \Sq^1 \bar{\iota}_3).
\end{align*}

In the first case a basis for $uns \mathcal{A}\{\iota_1\}$ is given by those $\Sq^{i_1} \cdots \Sq^{i_r} {\iota}_1$ which are admissible and have excess $\leq 1$, i.e.\ the $\Sq^{2^{i}} \cdots \Sq^{2} \Sq^1 {\iota}_1$, having degrees $1+1+2+2^2 + \cdots + 2^i$. This is only congruent to 1 modulo 4 for the class $\iota_1$ itself. The same goes for the summands generated by $\iota_1'$ and $\iota''_1$.

In the second case, $uns \mathcal{A}\{\iota_2\}/(\Sq^1 \iota_2)$ has basis given by those $\Sq^{i_1} \cdots \Sq^{i_r} {\iota}_2$ which are admissible, have excess $\leq 2$, and have $i_r \geq 2$, i.e.\ the $\Sq^{2^{i}} \cdots \Sq^{2^2} \Sq^2 {\iota}_2$. These all have even degree. On the other hand $uns \mathcal{A}\{\bar{\iota}_3\}/(\Sq^1 \bar{\iota}_3)$ has basis given by those $\Sq^{i_1} \cdots \Sq^{i_r} {\iota}_2$ which are admissible, have excess $\leq 3$, and have $i_r \geq 2$. These have the form $\Sq^{2^j(2^i+1)} \cdots \Sq^{2(2^i+1)} \Sq^{2^i+1} \Sq^{2^{i-1}} \cdots \Sq^{2} {\iota}_3$ for some $i \geq 1$ and $j \geq -1$ (correctly interpreted), and some checking of cases shows that such elements have degree congruent to 3 modulo 4 only in the case of ${\iota}_3$ itself.
\end{proof}

\begin{lemma}\label{lem:LowestClassEven}
If $n$ is even, so $(H_\bZ, \lambda)$ is symmetric, then on $\pi_1$ the map
$$B\Aut(H_\bZ, \lambda) \overset{inc}\lra \Omega^{\infty+2n}\LL^s(\bZ) \lra \tau_{\leq 1} \Omega^{\infty+2n}\LL^s(\bZ) = K(\bZ/2, 1)$$
assigns to an automorphism $\phi$ of $(H_\bZ, \lambda)$ the number $\dim_{\bZ/2} (\bZ/2 \otimes tors \tfrac{H_\bZ}{(\mathrm{Id}-\phi)H_\bZ})$ modulo 2. In particular this function $r_1$ is a homomorphism, and the map in question is the cohomology class given by $r_1$.
\end{lemma}
\begin{proof}
Given a $\phi \in \Aut(H_\bZ, \lambda)$, we are required to evaluate the de Rham invariant of the algebraic mapping torus $T_\phi$ of $\phi : H_\bZ[-n] \to H_{\bZ}[-n]$. Fortunately the value of this invariant does not depend on the symmetric structure but only on the underlying chain complex $T_\phi \simeq ( H_\bZ[-n] \xrightarrow{1-\phi} H_\bZ[-n-1] )$, and is given by 
$$\mathrm{dim}_{\bZ/2} (\bZ/2 \otimes  tors H^{-n-1}(T_\phi ; \bZ))= \dim_{\bZ/2} (\bZ/2 \otimes tors \tfrac{H_\bZ}{(\mathrm{Id}-\phi)H_\bZ}).$$
\end{proof}

\begin{remark}\label{rem:r1Calc}
Using the first fibration in \eqref{eq:GWQuillenSeq} and the calculation of the groups $H^1(\tau_{>0}\GW^s(\bZ ; +) ;\bZ/2)$ from the proof of Lemma \ref{lem:CohIsSmall}, we find that
\begin{equation}\label{eq:H1Even}
H_1(\Omega^\infty_0 \GW^s(\bZ ; +);\bZ) \cong \bZ/2 \oplus \bZ/2 \oplus \bZ/2.
\end{equation}
We may describe this as follows. 

Negation gives automorphisms $i_+$ and $i_-$ of $(\bZ, (1))$ and $(\bZ, (-1))$ respectively, and there is an automorphism $i_{rot}$ of the positive definite form $(\bZ^2, \left(\begin{smallmatrix}
1 & 0 \\
0 & 1
\end{smallmatrix}\right))$ given by the rotation $\left(\begin{smallmatrix}
0 & -1 \\
1 & 0
\end{smallmatrix}\right)$. In parallel we have the maps
$$({\det}^+, {\det}^-) : H_1(\Omega^\infty_0 \GW^s(\bZ ; +);\bZ) \lra H_1(\Omega^\infty_0 \GW^{top}(\bR ; +);\bZ) = \bZ/2 \oplus \bZ/2$$
which assign to an automorphism of a symmetric form the sign of the determinant of the induced map on positive and negative definite subspaces of its realification, as well as the de Rham invariant $r_1 : H_1(\Omega^\infty_0 \GW^s(\bZ ; +);\bZ) \to \bZ/2$ as described in the proof of Lemma \ref{lem:LowestClassEven}. It is easy to check that $r_1$ is nontrivial on each of $i_+, i_-, i_{rot}$, and that ${\det}^+(i_+)$ and ${\det}^-(i_-)$ are nontrivial but all other values of ${\det}^\pm$ on $i_+, i_-, i_{rot}$ are trivial. Thus $i_+, i_-, i_{rot}$ give a basis for the left-hand side of \eqref{eq:H1Even}, and ${\det}^+, {\det}^-, r_1$ give a dual basis for it.
\end{remark}

The following lemma finishes the proof of Theorem \ref{thm:FateofdR}.

\begin{lemma}\label{lem:LowestClassOdd}
If $n$ is odd then the composition
$$\Omega^\infty_0 \GW^s(\bZ; -) \lra \Omega^{\infty+2n} \LL^s(\bZ) \xrightarrow{\Omega^{\infty+2n} r} \Omega^{\infty+2n} \left(\bigoplus_{j \in \bZ} \H\bZ/2[4j+1] \right) \lra K(\bZ/2, 3)$$
given by projecting to the lowest degree summand is the unique nontrivial element of $H^3(\Omega^\infty_0 \GW^s(\bZ; -);\bZ/2) = H^3(BSp_\infty(\bZ);\bZ/2)$.
\end{lemma}
\begin{proof}
We implicitly 2-complete everywhere, and work in $\bZ/2$-cohomology. Firstly, taking infinite loop spaces of the second fibration in \eqref{eq:GWQuillenSeq} and looping it up gives a fibration
$$Sp \lra \Omega^\infty_0 \GW^s(\bZ; -) \lra BU,$$
which is pulled back from a fibration $Sp \to \Omega^\infty_0 \GW^s(\bF_3; -) \to BSp$. By \cite[Proposition I.4.2]{FP} the Serre spectral sequence for the latter fibration collapses, so that of the former does too. It follows that $H^3(\Omega^\infty_0 \GW^s(\bZ; -)) = \bZ/2$, and in terms of the notation introduced in the proof of Lemma \ref{lem:CohIsSmall} it is generated by $\Omega^\infty \bar{\iota}_3$. We need to show that the map in the statement of the lemma is the nontrivial element of this group.

By the Main Theorem of \cite{No9II} (and that fact that $\pi_1(\Sigma^{-2n}\LL^s(\bZ))=0$) there is a homotopy fibre sequence of 0-connected covers
\begin{equation}\label{eq:No9Seq}
\tau_{>0}((S^{2\sigma-2} \otimes \K(\bZ))_{hC_2}) \lra \tau_{>0}\GW^s(\bZ; -) \lra \tau_{>0}\Sigma^{-2n}\LL^s(\bZ)
\end{equation}
where $\sigma$ denotes the sign representation of $C_2$. On cohomology the right-hand map has the form
$$H^*(\tau_{> 0} \GW^{s}(\bZ; -)) \longleftarrow \mathcal{A}\{\iota_2\}/(\Sq^1 \iota_2) \oplus \mathcal{A}\{\iota_3\}$$
in degrees $\leq 3$, and in terms of the calculations in the proof of Lemma \ref{lem:CohIsSmall} we must have $\iota_2 \mapsto A \iota_2$ and $\iota_3 \mapsto B\bar{\iota}_3$ for some $A, B \in \bZ/2$. We need to show that $B=1$, which we shall do by analysing the cohomology of $\tau_{>0}((S^{2\sigma-2} \otimes \K(\bZ))_{hC_2})$ in low degrees.

The unit map $\Sph \to \K(\bZ)$ induces an isomorphism on homotopy groups in degrees $\leq 2$, and induces an injection $\bZ/24 = \pi_3(\Sph) \to \pi_3(\K(\bZ)) = \bZ/48$. Thus the map $(S^{2\sigma-2})_{hC_2} \to (S^{2\sigma-2} \otimes \K(\bZ))_{hC_2}$ has fibre $\mathrm{F}$ which is 1-connected and has $\pi_2(\mathrm{F})=\bZ/2$. Taking 0-connected covers gives a fibration sequence
\begin{equation}\label{eq:DefF}\
\mathrm{F} \lra \tau_{>0}((S^{2\sigma-2})_{hC_2}) \lra \tau_{>0}((S^{2\sigma-2} \otimes \K(\bZ))_{hC_2}).
\end{equation}

By the Thom isomorphism we have
$$H^*((S^{2\sigma-2})_{hC_2}) \cong H^*(BC_2) \cdot u = \bZ/2[x] \cdot u$$
with $\mathcal{A}$-module structure determined by $\Sq(u) = w(2\sigma-2) \cdot u = (1+x^2) \cdot u$ , the Cartan formula, and the usual action of Steenrod operations on $H^*(BC_2) = \bZ/2[x]$. In degrees $\leq 3$ this means that the only nontrivial operations are
$$\Sq^1 (xu) = x^2 u \quad\quad\quad \Sq^2 u = x^2 u  \quad\quad\quad \Sq^2(xu) = x^3 u.$$
Using the fibre sequence $\tau_{>0}((S^{2\sigma-2})_{hC_2}) \to (S^{2\sigma-2})_{hC_2} \to \H\bZ$ we deduce that $H^i(\tau_{>0}((S^{2\sigma-2})_{hC_2}))$ is trivial for $i=0$ and 1-dimensional for $i=1,2$. Combined with \eqref{eq:DefF}, we find that $H^i(\tau_{>0}((S^{2\sigma-2} \otimes \K(\bZ))_{hC_2}))$ is trivial for $i=0$, 1-dimensional for $i=1$, and at most 1-dimensional for $i=2$. Returning to \eqref{eq:No9Seq}, as $H^1(\tau_{>0}((S^{2\sigma-2} \otimes \K(\bZ))_{hC_2}))=\bZ/2$ it follows that $A=0$, but then we have an exact sequence
$$0 \to \bZ/2\{\iota_2\} \to H^2(\tau_{>0}((S^{2\sigma-2} \otimes \K(\bZ))_{hC_2})) \to \bZ/2\{\iota_3\} \overset{B \cdot -}\to \bZ/2\{\bar{\iota}_3\} \to \cdots$$
and as $H^2(\tau_{>0}((S^{2\sigma-2} \otimes \K(\bZ))_{hC_2}))$ has dimension at most 1 it follows that $B=1$, as required.
\end{proof}

The following lemma shows that the de Rham invariants vanish in the presence of a quadratic structure.

\begin{lemma}
If $(H_\bZ, \lambda)$ is a nondegenerate $(-1)^n$-symmetric form and $\mu : H_\bZ \to \bZ/(1-(-1)^n)$ is a quadratic refinement of it, then the composition
$$B\Aut(H_\bZ, \lambda, \mu) \lra B\Aut(H_\bZ, \lambda) \xrightarrow{inc} \Omega^{\infty+2n} \LL^s(\bZ) \xrightarrow{\Omega^{\infty+2n} r} \Omega^{\infty+2n} \left(\bigoplus_{j \in \bZ} \H\bZ/2[4j+1] \right)$$
is nullhomotopic.
\end{lemma}
\begin{proof}
The composition of the first two maps agrees with 
$$B\Aut(H_\bZ, \lambda, \mu) \overset{inc}\lra \Omega^{\infty+2n} \LL^q(\bZ) \lra \Omega^{\infty+2n} \LL^s(\bZ),$$
but by \cite[eq.\ (1.10)]{TaylorWilliams} the composition $\LL^q(\bZ) \to  \LL^s(\bZ) \overset{r}\to \bigoplus_{j \in \bZ} \H\bZ/2[4j+1]$ is null.
\end{proof}

\section{Examples and applications}

\subsection{Multiplicativity of the signature}\label{sec:MultSigPoinc}

Normal $L$-theory $\LL^n(\bZ)$ fits into a fibration sequence
$$\LL^q(\bZ) \lra \LL^s(\bZ) \lra \LL^n(\bZ),$$
where the first map encodes the forgetful map from quadratic to symmetric $L$-theory. Its homotopy groups are therefore
$$\pi_i(\LL^n(\bZ)) = \begin{cases}
\bZ/8 & i \equiv 0 \mod 4\\
\bZ/2 & i \equiv 1,3 \mod 4\\
0 & i \equiv 2 \mod 4.
\end{cases}$$
It is in fact a ring spectrum and $\LL^s(\bZ) \to \LL^n(\bZ)$ is a ring map, which may be seen by tracing through the definitions in \cite[p.\ 384]{RanickiPreprint}; as it is 2-local, it follows as in Section \ref{sec:LThyBG} that $\LL^n(\bZ)$ is an Eilenberg--MacLane spectrum. Taylor and Williams produce specific maps
\begin{align*}
\hat{L} : \LL^n(\bZ) &\lra \bigoplus_{i \in \bZ} \H\bZ/8[4i]\\
\hat{r} : \LL^n(\bZ) & \lra \bigoplus_{j \in \bZ} \H\bZ/2[4j+1]
\end{align*}
and $\Sigma k : \LL^n(\bZ) \to \bigoplus_{k \in \bZ} \H\bZ/2[4k+2]$ exhibiting it as an Eilenberg--MacLane spectrum, whose precompositions with $\LL^s(\bZ) \to \LL^n(\bZ)$ are $L \mod 8$, $r$, and 0 respectively. 

Ranicki has constructed \cite[p.\ 385]{RanickiPreprint} a normal signature map $\sigma^n$ such that
\begin{equation*}
\begin{tikzcd}
\MSTop \arrow[rd, "\sigma"] \rar& \Mock^P \dar{\sigma^P} \rar & \mathrm{MSG} \dar{\sigma^n}\\
 & \LL^s(\bZ) \rar& \LL^n(\bZ)
\end{tikzcd}
\end{equation*}
commutes up to homotopy; here $\mathrm{MSG}$ is the Thom spectrum of the universal spherical fibration over $BSG$, and is a model for the cobordism theory given by Quinn's normal spaces \cite{QuinnNormal, QuinnUnpub}. Furthermore, Taylor and Williams construct $\hat{L}$ so that $\hat{L} \circ \sigma^n$ corresponds under the Thom isomorphism to the characteristic class of stable spherical fibrations $l \in H^{4*}(BSG ; \bZ/8)$ constructed by Brumfiel--Morgan \cite[\S 8]{BrumfielMorgan}. Using this we can apply the results of the previous section to oriented fibrations with finite Poincar{\'e} fibre as follows. (It is likely that this can be extended to Poincar{\'e} fibres which are merely finitely-dominated, but we do not pursue this here.) 

\begin{proposition}\label{prop:PoincarLMod8}
If $F^d \to E \overset{\pi} \to B$ is an oriented fibration with finite Poincar{\'e} fibre and finitely-dominated base then it has a fibrewise Spivak normal fibration $\nu_\pi$, and
$$
\int_\pi l(\nu_\pi) = \begin{cases}
0 & d \text{ odd}\\
(\Omega^{1-(-1)^n} \widetilde{\mathrm{ph}})(\xi_\bR) & d \text{ even}
\end{cases} \in H^{4*-d}(B;\bZ/8).
$$
\end{proposition}
\begin{proof}
By pulling back to a finite CW-complex dominating $B$ we may assume without loss of generality that $B = |K|$ is the realisation of a finite semi-simplicial set. By induction on the simplices of $K$ we may construct a $U \subset |K| \times \bR^N$ and compatible maps
$$\psi\vert_\sigma : U\vert_\sigma := U \cap(\sigma \times \bR^N) \overset{\sim}\lra E\vert_\sigma$$
for each simplex $\sigma$ of $K$, such that each $U\vert_\sigma  \subset \sigma \times \bR^N$ is a codimension 0 compact submanifold with boundary $\partial (U\vert_\sigma) =  (U\vert_{\partial\sigma}) \cup \partial_0(U\vert_\sigma)$. The inclusions $U\vert_\tau \to U\vert_\sigma$ for $\tau \leq \sigma$ are equivalences, as $E\vert_\tau \to E\vert_\sigma$ are, and it follows from Poincar{\'e} duality and induction over simplices that (for $N$ large enough) the inclusions $\partial_0(U\vert_\tau) \to \partial_0(U\vert_\sigma)$ are equivalences too. As the inclusions $\partial_0(U\vert_\sigma) \to U\vert_\sigma$ have homotopy fibre a $(N-d-1)$-sphere when $\sigma$ is a 0-simplex, by Spivak's theorem \cite[Proposition 4.6]{Spivak}, it follows that the same is true for all simplices. Then
$$\partial_0 U := \bigcup_{\sigma \in K} \partial_0(U\vert_\sigma) \lra \bigcup_{\sigma \in K} U\vert_\sigma = U \simeq E$$
also has homotopy fibre a $(N-d-1)$-sphere so gives the required oriented spherical fibration $\nu_\pi$. Collapsing the complement of $U$ gives a Pontrjagin--Thom collapse map
$$|K|_+ \wedge S^N \lra U/\partial_0 U \simeq \mathrm{Th}(\nu_\pi \to E)$$
which can be composed with the Thomification of the map $E \to BSG$ classifying the oriented spherical fibration $\nu_\pi$ to obtain $\alpha : |K| \to \Omega^{\infty+d} \mathrm{MSG}$. The manifolds $U\vert_\sigma \subset \sigma \times \bR^N$ can inductively be equipped with singular chains representing $(\psi\vert_\sigma)^{-1}_*[E\vert_\sigma, E\vert_{\partial \sigma}] \in H_{d + \mathrm{dim}(\sigma)}(U\vert_\sigma, U\vert_{\partial\sigma} ; \bZ)$ in order to define a semi-simplicial map $f : K \to \mathrm{Mock}^P(d)$.

We can then form the diagram
\begin{equation*}
\begin{tikzcd}[column sep={0.5em}, ampersand replacement=\&]
{|K|} \arrow[rd, "f"] \arrow[rrd, bend left=20, "\alpha"] \arrow[d, "\phi"] \& \\
{\begin{cases}
B\mathrm{Aut}(H^n(F;\bR), \lambda) & d=2n\\
* & d=2n+1
\end{cases}} \arrow[dd, "inc"] \& {|\mathrm{Mock}^P(d)|} \arrow[r] \arrow[d, "\sigma^P"] \& \Omega^{\infty+d} \mathrm{MSG} \arrow[d, "\sigma^n"]\\
 \& \Omega^{\infty+d} \LL^s(\bZ) \arrow[r]\arrow[d, "L"] \arrow[ld] \& \Omega^{\infty+d} \LL^n(\bZ) \arrow[d, "\hat{L}"]\\
\Omega^{\infty+d} \LL^s(\bR) \arrow[r, "\sim"] \& \Omega^{\infty+d}(\bigoplus_{i \in \bZ} \H\bZ_{(2)}[4i]) \arrow[r] \& \Omega^{\infty+d}(\bigoplus_{i \in \bZ} \H\bZ/8[4i])
\end{tikzcd}
\end{equation*}
in which all regions commute up to homotopy tautologically, apart from the leftmost trapezium, which commutes up to homotopy by surgery below the middle dimension as explained in Section \ref{sec:SurgeryBelowMid}. The claimed formula then follows by combining this diagram with Theorem \ref{thm:2LocalPontrjaginFormula}.
\end{proof}

\begin{remark}
The fibrewise Spivak fibration $\nu_\pi$ exists without assuming that the fibre is finite; constructions are given in \cite[Proposition 1.8]{QuinnUnpub} and \cite[Addendum C]{KleinDualising}.
\end{remark}

\begin{remark}
Taylor and Williams do not settle what characteristic class $\hat{r} \circ \sigma^n$ corresponds to in $H^{4*+1}(BSG;\bZ/2)$ under the Thom isomorphism, but if we call it $\rho$ for now then there is a corresponding identity for $\int_\pi \rho(\nu_\pi)$. It would be interesting to determine what this $\rho$ is.
\end{remark}

As a consequence we obtain the following generalisation of the result explained in Remark \ref{rem:MultSig}. It generalises \cite[Theorem 7.2]{Korzeniewski} in that the base can be an arbitrary Poincar{\'e} complex, not necessarily simple.

\begin{corollary}\label{cor:MultSigPoinc}
Let $F^d \to E^{4k} \overset{\pi}\to B^{4k-d}$ be an oriented fibration of oriented Poincar{\'e} complexes with finite fibre. Then
$$\sigma(E) \equiv \sigma(B) \cdot \sigma(F) \mod 4.$$
\end{corollary}
\begin{proof}
Writing $\nu_X$ for the Spivak normal fibration of a Poincar{\'e} complex $X$, and $\nu_\pi$ for the fibrewise Spivak normal fibration of $\pi$, there is an identity $\nu_E \simeq \pi^*(\nu_B) * \nu_\pi$ (see \cite[Theorem I]{KleinDualising}). Using the properties of the characteristic class $l$ given in \cite[Theorem I]{BrumfielMorgan} we calculate
\begin{align*}
\sigma(E) & \equiv \int_E l(\nu_E) \mod 8\\
&\equiv \int_E \pi^*l(\nu_B) \cdot l(\nu_\pi) \mod 4\\
& \equiv \int_B l(\nu_B) \int_\pi l(\nu_\pi) \mod 4
\end{align*}
but by Proposition \ref{prop:PoincarLMod8} and the last part of Lemma \ref{lem:DefPHTilde} we have $\int_\pi l(\nu_\pi) \equiv \sigma(F) \mod 4$ and so the above becomes $\sigma(B) \cdot \sigma(F) \mod 4$.
\end{proof}

This line of reasoning will give, in principle, information about $\sigma(E) - \sigma(B) \cdot \sigma(F) \mod 8$. There will be additional terms coming from (i) the characteristic class $l$ not exactly being multiplicative modulo 8 \cite[Theorem I (iii)]{BrumfielMorgan}, and (ii) the classes $\widetilde{\mathrm{ph}}_i$ not necessarily vanishing mod 8. It would be interesting to compare this with the formula of \cite{Rovi}.

\subsection{Fibrewise Stiefel--Whitney classes}

As a further application of Proposition \ref{prop:PoincarLMod8} we have the following vanishing result for fibre integrals of fibrewise Stiefel--Whitney classes. 

\begin{corollary}\label{cor:FibSWClass}
If $F \to E \overset{\pi}\to B$ is an oriented fibration with finite Poincar{\'e} fibre, then
$$\int_\pi w_i(-\nu_\pi) = \begin{cases}
\chi(F) & i=d \\
0 & \text{ else}.
\end{cases}$$
\end{corollary}

This is obvious for fibre bundles, as the vertical tangent bundle has dimension $d$ and so $w_i(-\nu_\pi)$ vanishes for $i>d$. But for fibrations with Poincar{\'e} fibre it does not seem obvious.

\begin{proof}
The Brumfiel--Morgan class $l$ reduces modulo 2 to the square of the Wu class \cite[Theorem I (ii)]{BrumfielMorgan}, so as the $\widetilde{\mathrm{ph}}_i$ reduce to zero modulo 2 for $i>0$ we obtain from Proposition \ref{prop:PoincarLMod8} the identity $\int_\pi V(\nu_\pi)^2 \equiv \sigma(F) \mod 2$ which is also $\chi(F) \mod 2$. As in the proof of Theorem \ref{thm:2LocalPontrjaginFormula} we take the total Steenrod square of this to obtain
$$\chi(F) \mod 2 \equiv \Sq\left(\int_\pi V(\nu_\pi)^2\right) = \int_\pi w(-\nu_\pi)$$
as required.
\end{proof}

\subsection{Multiplicativity of the de Rham invariant}\label{sec:dRmult}

We can make a similar analysis of the de Rham invariant $d$, though it turns out to not be multiplicative in general: however, the failure to be multiplicative gives an interpretation of the cohomology classes $r_1$ and $r_3$ appearing in Theorem \ref{thm:FateofdR}. See \cite{AlexdR} for a special case, and \cite{BSdR} for a related result on the $\bR$-semicharacteristic.

\begin{proposition}\label{prop:dRmult}
Let $F^d \to E^{4k+1} \overset{\pi}\to B^{4k+1-d}$ be an oriented block bundle of oriented topological manifolds. If $d$ is odd then $d(E)=\sigma(B) \cdot d(F) $. If $d$ is even with monodromy $\phi : B \to B\Aut(H^{d/2}(F;\bZ)/tors, \lambda)$ then
$$d(E) = d(B) \cdot \sigma(F) + \int_B w(TB) \cdot \phi^*\Sq(r_{2-(-1)^{d/2}}).$$
\end{proposition}
\begin{proof}
We have $V(\nu_E) = 1 + V_2(\nu_E) + V_4(\nu_E) + \cdots + V_{2k}(\nu_E)$, so
$$d(E) = \int_E V_{2k}(\nu_E) \cdot \Sq^1 V_{2k}(\nu_E) = \int_E V(\nu_E) \cdot \Sq^1 V(\nu_E).$$
The stable isomorphism $TE \cong_s \pi^* TB \oplus T_\pi^s E$ gives $V(\nu_E) = \pi^*V(\nu_B) \cdot V(\nu_\pi)$, using which we can rewrite this as
\begin{align*}
&\int_E \pi^*(V(\nu_B) \cdot \Sq^1 V(\nu_B)) \cdot V(\nu_\pi)^2 + \pi^*(V(\nu_B)^2) \cdot V(\nu_\pi) \cdot \Sq^1 V(\nu_\pi)\\
&\quad= \int_B V(\nu_B) \cdot \Sq^1 V(\nu_B) \int_\pi V(\nu_\pi)^2 + \int_E \pi^*(V(\nu_B))^2 \cdot V(\nu_\pi) \cdot \Sq^1 V(\nu_\pi).
\end{align*}
As in the proof of Corollary \ref{cor:FibSWClass} we have $\int_\pi V(\nu_\pi)^2 \equiv \sigma(F) \mod 2$ and so the first term simplifies to $d(B) \cdot \sigma(F)$ (which is trivial unless $d \equiv 0 \mod 4$). Abbreviating $V_i := V_i(\nu_\pi)$ we have
\begin{align*}
V(\nu_\pi) \cdot \Sq^1 V(\nu_\pi) &= \sum_{i} V_{2i} \cdot \Sq^1 V_{2i} + \sum_{i < j} V_{2i} \cdot \Sq^1 V_{2j} + V_{2j} \cdot \Sq^1 V_{2i}\\
&= \sum_{i} V_{2i} \cdot \Sq^1 V_{2i} + \Sq^1 \left(\sum_{i < j} V_{2i} \cdot V_{2j}\right)
\end{align*}
and so the second term in the above equation can be written as
$$\int_B V(\nu_B)^2 \cdot \int_\pi \sum_{i} V_{2i} \cdot \Sq^1 V_{2i} + \int_E \Sq^1 \left(\pi^*(V(\nu_B))^2 \cdot \sum_{i < j} V_{2i} \cdot V_{2j}\right).$$
The latter term is zero, as $E$ is by assumption oriented so $\Sq^1$ into its top cohomology is zero. Furthermore $\int_\pi \sum_{i} V_{2i} \cdot \Sq^1 V_{2i}$ is the family de Rham class, so is equal to the scalar $d(F)$ if $d$ is odd by Theorem \ref{thm:LinkingTriv}, and is equal to the class $r_1$ or $r_3$ if $d$ is even by Theorem \ref{thm:FateofdR}. Combining this with the fact that $\int_B V(\nu_B)^2 \equiv \sigma(B) \mod 2$ if $B$ is even-dimensional, and with $\int_B V(\nu_B) \cdot - = \int_B \Sq(-)$, gives the claimed formulas.
\end{proof}

\begin{example}
Let $\mathbb{CP}^2 \to E^{4k+1} \overset{\pi}\to \mathbb{RP}^{4k-3}$ be the standard fibering of the Dold manifold, i.e.\ $E = S^{4k-3} \times_{C_2} \mathbb{CP}^2$ where the involution acts antipodally on $S^{4k-3}$ and by complex conjugation on $\mathbb{CP}^2$. The generator of $\pi_1(\mathbb{RP}^{4k-3}) = \bZ/2$ acts on the intersection form $(\bZ, (1))$ of $\mathbb{CP}^2$ by a sign, so it follows as in Remark \ref{rem:r1Calc} that $\phi^* r_1 = x \in \bZ/2[x]/(x^{4k-2}) = H^*(\mathbb{RP}^{4k-3};\bZ/2)$. We have
$$
\int_{\mathbb{RP}^{4k-3}} w(T\mathbb{RP}^{4k-3})  \cdot \Sq(x) = \int_{\mathbb{RP}^{4k-3}} (1+x)^{4k-2} \cdot (x+x^2) = 1,
$$
and so $d(E) = d(\mathbb{RP}^{4k-3})\cdot \sigma(\mathbb{CP}^2)+1 = 1$, as $\mathbb{RP}^{4k-3}$ has de Rham invariant 0.
\end{example}

In contradistinction with this example, the following shows that for topological \emph{fibre} bundles with fibres of dimension $2 \mod 4$ the invariant $\phi^*r_3$ is trivial (so e.g.\ the de Rham invariant is multiplicative for such bundles, by Proposition \ref{prop:dRmult}).

\begin{proposition}
If $F^{4k+2} \to E \overset{\pi}\to B$ is an oriented topological fibre bundle, with $\phi : B \to BSp_{2g}(\bZ)$ classifying the associated local system of symplectic forms, then $\phi^*r_3=0 \in H^3(B;\bZ/2)$.
\end{proposition}
\begin{proof}
Writing $V(\nu_\pi) = 1 + {V}_2 + {V}_4 + V_6 + \cdots$, by Theorem \ref{thm:FateofdR} and the Family Signature Theorem we have $\phi^*r_3 = \int_\pi V_{2k+2} \cdot \Sq^1 V_{2k+2}$. Without loss of generality we may suppose that $B$ is a (not necessarily orientable) 3-manifold, whereupon we wish to show that $\int_B \phi^*r_3 = \int_E V_{2k+2} \cdot \Sq^1 V_{2k+2}$ vanishes. 

The Wu class of $B$ has the form $V(\nu_B) = 1 + x$ for $x$ of degree 1, so writing $V(\nu_E) = 1 + \bar{V}_1 + \bar{V}_2 + \cdots + \bar{V}_{2k+2}$ the identity $V(\nu_E) = V(\nu_B) \cdot V(\nu_\pi)$ implies that $\bar{V}_{2i} = V_{2i}$ and $\bar{V}_{2i+1} = x \cdot V_{2i}$, so in particular $V_{2i}=0$ for $i>k+1$. Thus
$$w(T_\pi E) = \Sq(V(\nu_\pi)) = 1 + (V_2) + (\Sq^1 V_2) + (V_4 + V_2^2) + \cdots + (\Sq^{2k+1} V_{2k+2}) + (V_{2k+2}^2)$$
and as $T_\pi E$ has dimension $4k+2$ it follows that $0= w_{4k+3}(T_\pi E) = \Sq^{2k+1} V_{2k+2}$. Using the Adem relation $$\Sq^2 \Sq^{2k+1} = \Sq^{2k+2} \Sq^1 + \begin{cases}
\Sq^{2k+3} & k \text{ odd}\\
0 & k \text{ even}
\end{cases}$$
and instability, we deduce that $\Sq^{2k+2} \Sq^1 V_{2k+2}=0$ too. But then by definition of the Wu classes of $E$ and the relation $\bar{V}_{2k+2}={V}_{2k+2}$ we find that
$$0 = \int_E \Sq^{2k+2} \Sq^1 V_{2k+2} = \int_E \bar{V}_{2k+2} \Sq^1 V_{2k+2} = \int_E V_{2k+2} \Sq^1 V_{2k+2}$$
as required.
\end{proof}

\subsection{Integrality}
For an oriented topological block bundle $\pi : E \to |K|$ with $d$-dimensional fibres the discussion in Section \ref{sec:AddendaEven} shows that there is a 2-integral refinement $\mathcal{L}(T_\pi^s E) \in H^{*}(E ; \bZ_{(2)})$ of the Hirzebruch $L$-class, and that its fibre integral satisfies
$$\int_\pi \mathcal{L}(T_\pi^s E) = \begin{cases}
 \widetilde{\mathrm{ph}}(\phi^* \xi_\bR) & d\text{ is even}\\
0 & d\text{ is odd}
\end{cases} \in H^{*}(|K|;\bZ_{(2)}).$$

At odd primes $p$ a similar result is available in a range of degrees: for this we can work either with $\KO[\tfrac{1}{2}]$ or $\LL^s(\bZ)$; lets take the former for concreteness. As the first torsion in $\pi_*(S_{(p)})$ is in degree $2p-3$, it follows from the Atiyah--Hirzebruch spectral sequence that there is a unique homotopy class
$$\tau_{\geq 0} \KO_{(p)} \lra \bigoplus_{k = 0}^{\lfloor (2p-4)/4 \rfloor} \H\bZ_{(p)}[4i]$$
which on homotopy groups sends $a^k$ to 1 for $0 \leq k \leq \lfloor (2p-4)/4 \rfloor$. Pulled back along the Sullivan orientation $\Delta_{Top}$ these cohomology classes correspond under the Thom isomorphism to canonical $p$-local classes $\mathcal{L}_k \in H^{4k}(BSTop ; \bZ_{(p)})$ defined for $4k < 2p-3$, whose rationalisations are the usual topological $L$-classes. If $d=2n$ then looped $2n$ times and pulled back along $\sign : B\Aut(H_\bR, \lambda) \to \Omega^{\infty+2n} \KO_{(p)}$ these classes correspond to $2^{2k-n} \mathrm{ch}_{2k-n}(\xi)$ by Theorem \ref{thm:KOSigFormula} and its proof (note that $(2k-n)!$ is a $p$-local unit as $2k < p-1$, so $\mathrm{ch}_{2k-n}$ is indeed defined $p$-integrally). Theorem \ref{thm:FSTSullivan} then gives the $p$-integral identity
$$\int_\pi \mathcal{L}_{k}(T_\pi^s E) = \begin{cases}
2^{2k-n} \mathrm{ch}_{2k-n}(\phi^*\xi) & d\text{ is even}\\
0 & d\text{ is odd}
\end{cases} \in H^{*}(|K|;\bZ_{(p)})$$
for $4k < 2p-3$.

\bibliographystyle{amsalpha}
\bibliography{biblio}

\newcommand{\etalchar}[1]{$^{#1}$}
\providecommand{\bysame}{\leavevmode\hbox to3em{\hrulefill}\thinspace}
\providecommand{\MR}{\relax\ifhmode\unskip\space\fi MR }
\providecommand{\MRhref}[2]{%
  \href{http://www.ams.org/mathscinet-getitem?mr=#1}{#2}
}
\providecommand{\href}[2]{#2}
\begin{thebibliography}{HLLRW21}

\bibitem[Ale81]{AlexdR}
J.~C. Alexander, \emph{On the de {R}ham invariant of a fibered
  {$(4r+1)$}-dimensional orientable manifold}, Math. Ann. \textbf{256} (1981),
  no.~4, 429--437.

\bibitem[AS68]{AtiyahSingerIII}
M.~F. Atiyah and I.~M. Singer, \emph{The index of elliptic operators. {III}},
  Ann. of Math. (2) \textbf{87} (1968), 546--604.

\bibitem[AS69]{AtiyahSegal}
M.~F. Atiyah and G.~B. Segal, \emph{Equivariant {$K$}-theory and completion},
  J. Differential Geometry \textbf{3} (1969), 1--18.

\bibitem[Ati69]{Atiyah}
M.~F. Atiyah, \emph{The signature of fibre-bundles}, Global {A}nalysis
  ({P}apers in {H}onor of {K}. {K}odaira), Univ. Tokyo Press, Tokyo, 1969,
  pp.~73--84.

\bibitem[BK05]{BK}
A.~J. Berrick and M.~Karoubi, \emph{Hermitian {$K$}-theory of the integers},
  Amer. J. Math. \textbf{127} (2005), no.~4, 785--823.

\bibitem[BM76]{BrumfielMorgan}
G.~W. Brumfiel and J.~W. Morgan, \emph{Homotopy theoretic consequences of {N}.
  {L}evitt's obstruction theory to transversality for spherical fibrations},
  Pacific J. Math. \textbf{67} (1976), no.~1, 1--100.

\bibitem[Bro72]{BrowderSpivak}
W.~Browder, \emph{Poincar\'{e} spaces, their normal fibrations and surgery},
  Invent. Math. \textbf{17} (1972), 191--202.

\bibitem[BRS76]{BRS}
S.~Buoncristiano, C.~P. Rourke, and B.~J. Sanderson, \emph{A geometric approach
  to homology theory}, London Mathematical Society Lecture Note Series, No. 18,
  Cambridge University Press, Cambridge-New York-Melbourne, 1976.

\bibitem[BS82]{BSdR}
J.~C. Becker and R.~E. Schultz, \emph{The real semicharacteristic of a fibered
  manifold}, Quart. J. Math. Oxford Ser. (2) \textbf{33} (1982), no.~132,
  385--403.

\bibitem[CDH{\etalchar{+}}20a]{No9II}
B.~Calm{\`e}s, E.~Dotto, Y.~Harpaz, F.~Hebestreit, M.~Land, K.~Moi, D.~Nardin,
  T.~Nikolaus, and W.~Steimle, \emph{Hermitian {$K$}-theory for stable
  $\infty$-categories {II}: {C}obordism categories and additivity},
  arXiv:2009.07224, 2020.

\bibitem[CDH{\etalchar{+}}20b]{No9III}
\bysame, \emph{Hermitian {$K$}-theory for stable $\infty$-categories {III}:
  {G}rothendieck--{W}itt groups of rings}, arXiv:2009.07225, 2020.

\bibitem[CDH{\etalchar{+}}23]{No9I}
\bysame, \emph{Hermitian {$K$}-theory for stable $\infty$-categories {I}:
  {F}oundations}, Sel. Math. New Ser. \textbf{29} (2023), no.~1, Paper No. 10.

\bibitem[ER22]{EbertReinhold}
J.~Ebert and J.~Reinhold, \emph{Some rational homology computations for
  diffeomorphisms of odd-dimensional manifolds}, arXiv:2203.03414, 2022.

\bibitem[FM92]{FreedMelrose}
D.~S. Freed and R.~B. Melrose, \emph{A mod {$k$} index theorem}, Invent. Math.
  \textbf{107} (1992), no.~2, 283--299.

\bibitem[FP78]{FP}
Z.~Fiedorowicz and S.~Priddy, \emph{Homology of classical groups over finite
  fields and their associated infinite loop spaces}, Lecture Notes in
  Mathematics, vol. 674, Springer, Berlin, 1978.

\bibitem[FQ90]{FreedmanQuinn}
M.~H. Freedman and F.~Quinn, \emph{Topology of 4-manifolds}, Princeton
  Mathematical Series, vol.~39, Princeton University Press, Princeton, NJ,
  1990.

\bibitem[Got79]{Gottlieb}
D.~H. Gottlieb, \emph{Poincar\'{e} duality and fibrations}, Proc. Amer. Math.
  Soc. \textbf{76} (1979), no.~1, 148--150.

\bibitem[GRW22]{GRWPontAlgInd}
S.~Galatius and O.~Randal-Williams, \emph{Algebraic independence of topological
  {P}ontryagin classes}, arXiv:2208.11507, 2022.

\bibitem[HKR07]{HKR}
I.~Hambleton, A.~Korzeniewski, and A.~Ranicki, \emph{The signature of a fibre
  bundle is multiplicative mod 4}, Geom. Topol. \textbf{11} (2007), 251--314.

\bibitem[HLLRW21]{HLLRW}
F.~Hebestreit, M.~Land, W.~L\"{u}ck, and O.~Randal-Williams, \emph{A vanishing
  theorem for tautological classes of aspherical manifolds}, Geom. Topol.
  \textbf{25} (2021), no.~1, 47--110.

\bibitem[HLN21]{HLN}
F.~Hebestreit, M.~Land, and T.~Nikolaus, \emph{On the homotopy type of
  {L}-spectra of the integers}, J. Topol. \textbf{14} (2021), no.~1, 183--214.

\bibitem[HS21]{HS}
F.~Hebestreit and W.~Steimle, \emph{Stable moduli spaces of hermitian forms},
  arXiv:2103.13911, 2021.

\bibitem[Kar78]{KaroubiBook}
M.~Karoubi, \emph{{$K$}-theory}, Grundlehren der Mathematischen Wissenschaften,
  Band 226, Springer-Verlag, Berlin-New York, 1978, An introduction.

\bibitem[Kar84]{KaroubiRig}
\bysame, \emph{Relations between algebraic {$K$}-theory and {H}ermitian
  {$K$}-theory}, Proceedings of the {L}uminy conference on algebraic
  {$K$}-theory ({L}uminy, 1983), vol.~34, 1984, pp.~259--263.

\bibitem[Kle01]{KleinDualising}
J.~R. Klein, \emph{The dualizing spectrum of a topological group}, Math. Ann.
  \textbf{319} (2001), no.~3, 421--456.

\bibitem[KMM13]{KMM}
P.~K\"{u}hl, T.~Macko, and A.~Mole, \emph{The total surgery obstruction
  revisited}, M\"{u}nster J. Math. \textbf{6} (2013), no.~1, 181--269.

\bibitem[Kor05]{Korzeniewski}
A.~Korzeniewski, \emph{On the signature of fibre bundles and absolute
  {W}hitehead torsion}, Ph.D. thesis, University of Edinburgh, 2005,
  \url{https://www.maths.ed.ac.uk/~v1ranick/surgery/korzen.pdf}.

\bibitem[KS77]{KS}
R.~C. Kirby and L.~C. Siebenmann, \emph{Foundational essays on topological
  manifolds, smoothings, and triangulations}, Annals of Mathematics Studies,
  No. 88, Princeton University Press, Princeton, N.J.; University of Tokyo
  Press, Tokyo, 1977, With notes by J.\ Milnor and M.\ Atiyah.

\bibitem[Lan76]{Landweber}
P.~S. Landweber, \emph{Homological properties of comodules over {$MU\sb\ast
  (MU)$} and {BP{$\sb\ast $}}({BP})}, Amer. J. Math. \textbf{98} (1976), no.~3,
  591--610.

\bibitem[Lan22]{MarkusNote}
M.~Land, \emph{Reducibility of low-dimensional {P}oincar\'{e} duality spaces},
  M\"{u}nster J. Math. \textbf{15} (2022), no.~1, 47--81.

\bibitem[Lev72]{LevittPoincare}
N.~Levitt, \emph{Poincar\'{e} duality cobordism}, Ann. of Math. (2) \textbf{96}
  (1972), 211--244.

\bibitem[LM14]{LauresMcClure}
G.~Laures and J.~E. McClure, \emph{Multiplicative properties of {Q}uinn
  spectra}, Forum Math. \textbf{26} (2014), no.~4, 1117--1185.

\bibitem[LMP69]{LMP}
G.~Lusztig, J.~Milnor, and F.~P. Peterson, \emph{Semi-characteristics and
  cobordism}, Topology \textbf{8} (1969), 357--359.

\bibitem[LR92]{LuckRanicki}
W.~L\"{u}ck and A.~Ranicki, \emph{Surgery obstructions of fibre bundles}, J.
  Pure Appl. Algebra \textbf{81} (1992), no.~2, 139--189.

\bibitem[Mey72]{Meyer}
W.~Meyer, \emph{Die {S}ignatur von lokalen {K}oeffizientensystemen und
  {F}aserb\"{u}ndeln}, Bonn. Math. Schr. (1972), no.~53, viii+59.

\bibitem[MM79]{MadsenMilgram}
I.~Madsen and R.~J. Milgram, \emph{The classifying spaces for surgery and
  cobordism of manifolds}, Annals of Mathematics Studies, No. 92, Princeton
  University Press, Princeton, N.J.; University of Tokyo Press, Tokyo, 1979.

\bibitem[MS74]{MorganSullivan}
J.~W. Morgan and D.~P. Sullivan, \emph{The transversality characteristic class
  and linking cycles in surgery theory}, Ann. of Math. (2) \textbf{99} (1974),
  463--544.

\bibitem[Qui71]{QuinnUnpub}
F.~Quinn, \emph{Surgery on {P}oincar\'{e} spaces}, Mimeographed Notes, NYU
  \url{https://www.maths.ed.ac.uk/~v1ranick/papers/quinnpoi.pdf}, 1971.

\bibitem[Qui72a]{QuillenFF}
D.~Quillen, \emph{On the cohomology and {$K$}-theory of the general linear
  groups over a finite field}, Ann. of Math. (2) \textbf{96} (1972), 552--586.

\bibitem[Qui72b]{QuinnNormal}
F.~Quinn, \emph{Surgery on {P}oincar\'{e} and normal spaces}, Bull. Amer. Math.
  Soc. \textbf{78} (1972), 262--267.

\bibitem[Qui95]{QuinnAd}
\bysame, \emph{Assembly maps in bordism-type theories}, Novikov conjectures,
  index theorems and rigidity, {V}ol. 1 ({O}berwolfach, 1993), London Math.
  Soc. Lecture Note Ser., vol. 226, Cambridge Univ. Press, Cambridge, 1995,
  pp.~201--271.

\bibitem[Ran78]{RanickiPreprint}
A.~A. Ranicki, \emph{The algebraic theory of surgery},
  \url{https://www.maths.ed.ac.uk/~v1ranick/papers/ats.pdf}, 1978.

\bibitem[Ran79]{RanickiTSO}
\bysame, \emph{The total surgery obstruction}, Algebraic topology, {A}arhus
  1978 ({P}roc. {S}ympos., {U}niv. {A}arhus, {A}arhus, 1978), Lecture Notes in
  Math., vol. 763, Springer, Berlin, 1979, pp.~275--316.

\bibitem[Ran81]{RanickiPhoneBook}
\bysame, \emph{Exact sequences in the algebraic theory of surgery},
  Mathematical Notes, vol.~26, Princeton University Press, Princeton, N.J.;
  University of Tokyo Press, Tokyo, 1981.

\bibitem[Ran92]{RanickiBook}
\bysame, \emph{Algebraic {$L$}-theory and topological manifolds}, Cambridge
  Tracts in Mathematics, vol. 102, Cambridge University Press, Cambridge, 1992.

\bibitem[Ros03]{Rosenberg}
J.~Rosenberg, \emph{Groupoid {$C^*$}-algebras and index theory on manifolds
  with singularities}, Geom. Dedicata \textbf{100} (2003), 65--84.

\bibitem[Rov18]{Rovi}
C.~Rovi, \emph{The nonmultiplicativity of the signature modulo 8 of a fibre
  bundle is an {A}rf-{K}ervaire invariant}, Algebr. Geom. Topol. \textbf{18}
  (2018), no.~3, 1281--1322.

\bibitem[Spi67]{Spivak}
M.~Spivak, \emph{Spaces satisfying {P}oincar\'{e} duality}, Topology \textbf{6}
  (1967), 77--101.

\bibitem[Sto63]{StongBO}
R.~E. Stong, \emph{Determination of {$H^{\ast} ({\rm BO}(k,\cdots,\infty
  ),Z_{2})$} and {$H^{\ast} ({\rm BU}(k,\cdots,\infty ),Z_{2})$}}, Trans. Amer.
  Math. Soc. \textbf{107} (1963), 526--544.

\bibitem[Sul05]{SullivanMIT}
D.~P. Sullivan, \emph{Geometric topology: localization, periodicity and
  {G}alois symmetry}, $K$-Monographs in Mathematics, vol.~8, Springer,
  Dordrecht, 2005, The 1970 MIT notes, Edited and with a preface by A.\
  Ranicki.

\bibitem[TW79]{TaylorWilliams}
L.~Taylor and B.~Williams, \emph{Surgery spaces: formulae and structure},
  Algebraic topology, {W}aterloo, 1978 ({P}roc. {C}onf., {U}niv. {W}aterloo,
  {W}aterloo, {O}nt., 1978), Lecture Notes in Math., vol. 741, Springer,
  Berlin, 1979, pp.~170--195.

\bibitem[Wal67]{WallPoincare}
C.~T.~C. Wall, \emph{Poincar\'{e} complexes. {I}}, Ann. of Math. (2)
  \textbf{86} (1967), 213--245.

\bibitem[Wal70]{WallBook}
\bysame, \emph{Surgery on compact manifolds}, London Mathematical Society
  Monographs, No. 1, Academic Press, London-New York, 1970.

\end{thebibliography}

\end{document}